\documentclass[12pt,oneside,reqno]{amsart}
\usepackage{txfonts}
\usepackage{bbm}
\usepackage{amsmath}
\usepackage{graphicx}
\usepackage{mathrsfs}
\usepackage{stmaryrd}
\usepackage{enumerate,amsmath,amssymb,amsthm}
\numberwithin{equation}{section}
\usepackage{amsthm}
\usepackage{amssymb,amsfonts,mathrsfs}
\usepackage{enumerate}
\usepackage{color}
\usepackage{xcolor}
\usepackage[colorlinks=true]{hyperref}
\hypersetup{
	linkcolor=blue,          
	citecolor=red,        
	filecolor=blue,      
	urlcolor=cyan
}

\theoremstyle{definition}
\newtheorem{theorem}{Theorem}[section]
\newtheorem{lemma}[theorem]{Lemma}%
\newtheorem{definition}[theorem]{Definition}%
\newtheorem{proposition}[theorem]{Proposition}%
\newtheorem{corollary}[theorem]{Corollary}
\newtheorem*{remark}{Remark}%
\def\mbb{\mathbb}
\def\mcl{\mathcal}
\def\mrm{\mathrm}
\def\msr{\mathscr}
\def\e{{\mathrm{e}}}
\def\d{\mathrm{d}}
\def\P{\mathbb{P}}
\def\E{\mathbb{E}}
\def\S{\mathbb{S}}

\allowdisplaybreaks
\begin{document}
	
	\title{Ergodicity for regime-switching neutral stochastic functional differential equations with infinite delay$^{\star}$}
	
	\date{}
	\author{Zuozheng Zhang$^{1}$, Fubao Xi$^{2,*}$}
	
	\thanks{$^{\star}$  Supported in part by the National Natural Science Foundation of China under Grant No. 12071031.}
	\thanks{$^{*}$ Corresponding author.}
	\thanks{E-mail addresses: zuozhengzhang@mail.bnu.edu.cn(Z. Zhang),  xifb@bit.edu.cn (F. Xi)}	
	
	\dedicatory{$^1$ School of Mathematics and Statistics, Wuhan University, Wuhan, Hubei 430072, P.R.China\\$^2$
		Institute for Mathematics and Interdisciplinary Sciences, Beijing Institute of Technology, Zhuhai, 519088, P.R.China}
	
	\begin{abstract}
		This work focuses on a class of regime-switching neutral stochastic functional differential equations (RNSFDEs) with infinite delay, in which the switching component can possess finite or countably infinite many states. To ensure the well-posedness of the underlying process, we first investigate the well-posedness for NSFDEs without Markovian switching under dissipativity conditions, and obtain the desired result by Skorohod's representation. By utilizing the moment estimate of exponential functionals of the switching component, we derive the exponential ergodicity in Wasserstein distance for RNSFDEs with a finite state space using the coupling method. To address the difficulty posed by the infinite state space, we obtain the same exponential ergodicity  by applying the finite partition method along with Lyapunov functions and $M$-matrix theory. 
	\end{abstract}
	\maketitle
	\rm
	\section{Introduction}
	In recent years, the ergodicity of stochastic functional differential equations (SFDEs) has attracted extensive attention, as the future state of such equations depends not only on the current state but also on the historical states of the system. It is well-known that the solutions to these processes no longer possess the Markov property; instead, the solution mapping (i.e., the segment process of solutions) forms a Markov process in the path space.  It is important to note that when the noise depends on the path, the transition probabilities of solution mappings starting from different initial values are always mutually singular. In particular, the corresponding Markov semigroup lacks the strong Feller property, and what is more, the only "small sets" in this system are singletons consisting of individual points. Effective tools for studying ergodicity, such as functional inequalities, Lyapunov-type criteria, Harris’ theorem, and coupling methods, are difficult to apply to these infinite-dimensional Markov processes. 
	Research on the ergodicity and related properties, including gradient estimates and Harnack-type inequalities, of SFDEs has seen the emergence of a series of new methodologies. Relevant advances can be found in \cite{bao2013derivative,BAO2019SPA,butkovsky2014subgeometric,OLEG2020AAP,butkovsky2017invariant,hairer2011asymptotic,wang2011harnack,wang2022stochastic}.

	However, the existing research still lacks systematic analysis of neutral stochastic functional differential equations (NSFDEs) with delay terms involving derivatives. In practical dynamic systems, the current state often depends on the rate of change of historical states, and neutral equations can accurately characterize this path dependence on the historical rate of change by introducing neutral terms containing delayed derivatives. Therefore, carrying out relevant research is of important theoretical significance and practical application value. In recent years, this field has made a series of advances in numerical analysis, stability, and ergodicity:  Wu and Mao \cite{wu2008numerical} rigorously proved the strong mean-square convergence of Euler–Maruyama numerical solutions for NSFDEs with delayed derivative terms under local Lipschitz conditions and linear growth assumptions, by developing analytical techniques adapted to the structure of neutral terms. Zong and Wu \cite{zong2016exponential} systematically established criteria for exponential stability in \(p\)th moment and almost sure exponential stability for both exact and numerical solutions of NSFDEs, using Lyapunov function methods combined with probabilistic tools such as the Chebyshev inequality and Borel–Cantelli lemma. Their work revealed the necessity of linear growth conditions for numerical methods to preserve stability and the advantages of the backward Euler–Maruyama strategy in handling non-uniform Lipschitz coefficients.  Bao et al. \cite{bao2020ergodicity} focused on NSFDEs with infinite delays, interpreting the weak Harris theorem through Wasserstein coupling methods and proving exponential ergodicity of functional solutions under the Wasserstein distance under appropriate dissipative conditions. This provides a new theoretical framework for analyzing the long-term asymptotic behavior of neutral systems with infinite delays. For more related research, see literature such as \cite{BaoYinYuanBook2016,boufoussi2012neutral,mahmudov2006existence,mao2006stochastic,zhu2025approximation}.

	Although SFDEs effectively describe time-delay phenomena in complex systems, they have limitations in characterizing scenarios where system operational mechanisms change significantly with internal and external conditions. Therefore, studying regime-switching SFDEs is of great significance. Currently, systematic research on the ergodicity of regime-switching SFDEs remains relatively limited. Below, we briefly summarize key recent advances:  Li and Xi \cite{li2021convergence} were the first to establish criteria for solution convergence, uniform boundedness, and exponential ergodicity under the Wasserstein distance for Markov-switching functional diffusion processes with infinite delays under dissipative conditions. Shi et al. \cite{shi2022ergodicity} further strengthened the conclusion of exponential ergodicity for the same model through improved coupling methods under appropriate conditions. Based on the uniform mixing properties of solution processes, they derived a strong law of large numbers for additive functionals, significantly expanding the theoretical analysis dimensions of the model.  Bao et al. \cite{BAO2023NONLINEAR} focused on functional diffusion processes with Markov switching, proving the existence of a unique invariant probability measure and that solution processes converge to the equilibrium state at an exponential rate under the Wasserstein distance.   Zhai and Xi \cite{zhai2024exponential} further extended the research to multiple classes of SFDEs with Markovian switching and finite delays, including classical SFDEs, neutral SFDEs, and Lévy-driven jump-diffusion processes. 
	
	In this paper, we inverstigate the ergodicity for a class of regime-switching neutral stochastic functional differential equations (RNSFDEs) with infinite delay. Let ${C\left((-\infty, 0] ;\mbb R^d\right)}$ denote the family of all continuous functions from $(-\infty, 0]$ to  $\mbb R^d$. For a given $r>0$, set
	\[
	C_r=\left\{\varphi\in C((-\infty, 0] ;\mbb R^d);\lim_{\theta\to-\infty }\e^{r\theta}\varphi(\theta) \text{ exists in }\mbb R^d\right\},
	\]
	which is a Banach space with norm $\|\varphi\|_r=\sup_{-\infty<\theta\leq0}\e^{r\theta}|\varphi(\theta)|$. Let $\textbf{0}$ be the zero element of the space $C_r$.  Let $\S = \{1, 2, \cdots , N\}$, $1 \leq  N \leq  \infty$, be the state space.
	Consider 
	\begin{equation}\label{eq:fangcheng1}
		\mathrm{d} \{X(t)-G(X_t,\Lambda(t))\}=b\left(X_t, \Lambda(t)\right) \mathrm{d} t+\sigma\left(X_t, \Lambda(t)\right) \mathrm{d} W(t), \quad t \geq 0,
	\end{equation}
	with initial data $(X_0,\Lambda(0))=(\xi,i) \in C_r\times\mbb S  $ , where $G$, $b: C_r \times \mbb S \rightarrow \mathbb{R}^d$ and  $\sigma: C_r \times \mbb S \rightarrow \mathbb{R}^{d\times d}$ are Borel measurable functions. $X_t(\theta):=X(t+\theta)$, $-\infty<\theta \leq 0$, is the segment process of $X(t)$ and $(W(t))$ is an $d$-dimensional Brownian motion on a complete filtration probability space $\left(\Omega, \mathscr{F},\left(\mathscr{F}_t\right)_{t \geq 0}, \mathbb{P}\right)$. The component $(\Lambda(t))$ is a right continuous stochastic process taking values in the state space $\mbb S$ with generator $Q=(q_{kl})_{N\times N}$ given by
	\begin{equation}\label{eq:fangcheng2}
		\mbb {P}\{\Lambda(t+\Delta)=l \mid \Lambda(t)=k \}= \begin{cases}q_{kl} \Delta+o(\Delta), & l \neq k, \\ 1+q_{kk} \Delta+o(\Delta), & l=k,\end{cases}
	\end{equation}
	provided $\Delta \downarrow 0$, where $q_{kl} \geq 0$ is the transition rate from $k$ to $l$, if $l \neq k$; while $q_{kk}=-\sum_{l \neq k} q_{kl}$.  We assume that the Brownian motion $(W(t))$  is independent of the Markov chain $(\Lambda(t))$. We further assume that the Markov chain $(\Lambda(t))$ is irreducible, which means that the system can switch from each regime to any other regime.
	
	In this paper, we focus on the RNSFDE $\left(X_t, \Lambda(t)\right)$ defined by \eqref{eq:fangcheng1} and \eqref{eq:fangcheng2}. The aim of this paper is to investigate the exponential ergodicity in Wasserstein distance of the underlying process by using the coupling method.  In this paper,  we first give the existence and uniqueness of solution maps for the RNSFDE  in Section \ref{sec:3}.  To this aim, for $p=2$, we study the well-posedness for  NSFDEs without Markovian switching Eq. \eqref{eq:NSFGE} under the dissipativity condition. Based on  this result, we derive the main result in this section by virtue of Skorokhod’s representation \eqref{Eq:1026:1} of the switching process $(\Lambda(t))$.

	In Section \ref{sec:4}, we consider the exponential ergodicity in Wasserstein distance for the RNSFDE in a finite state space. To this aim, we introduce an auxiliary process $Y^{\xi, i}(t):=X^{\xi, i}(t)-G(X_t^{\xi, i},\Lambda^{i}(t))$. Based on the esitmate  of the exponential functional of \( f(\Lambda(t)) \), we shall demonstrate the uniform boundedness of  $(Y(t))$ and establish the $L^2(\Omega;\mathbb{R}^{d})$-convergence of two processes $(Y^{\xi, i}(t))$ and $(Y^{\eta, i}(t))$ originating from different initial data. Thanks to these careful calculations, we extend these results to the solution map $(X_t)$. Finally, we use the coupling method to derive the exponential ergodicity in Wasserstein distance.
	
	In Section \ref{sec:wuqiong}, we proceed to study the exponential ergodicity for the RNSFDE in an infinite state space. When $\S$ is a infinite set, Lemma \ref{eq:wuwei} does not hold. Hence, we can not follow the method as in Sectin \ref{sec:4}. Here, we use the finite partition approach developed in \cite{SHAO2015SPA}. More precisely, the infinite state space $\S$ is divided into finite subsets, and a new Markov chain in a finite state space is constructed. With the aid of the new Markov chain, the uniform boundedness and the convergence for solution maps are also established by ultilizing the theory of Lyapunov functions and $M$-matrices. Following the argument as in Theorem \ref{thm:first}, we also obtain the exponential ergodicity in  the infinite state space.
	
	This paper is organized as follows. Section \ref{sec:2} introduces some necessary notations, notions and assumptions. 
	Section \ref{sec:3}  establishes the existence and uniqueness of solution maps for the RNSFDE. Section \ref{sec:4} is  devoted to establishing the uniform boundedness and the convergence for solution maps, and the exponential ergodicity for the RNSFDE in a finite state space. Section \ref{sec:wuqiong} proceeds to the study of  the RNSFDE in an infinite state and derives a similar exponential ergodicity result.
	
	Throughout the paper, C with or without indexes will denote different positive constants (depending on the indexes) whose values are not important. 
	\section{Preliminaries}\label{sec:2}
	We first introduce some notations throughout this paper.  $\mbb R^d$ stands for the $d$-dimensional Euclidean space. We will use $|\cdot|$ and $\langle\cdot,\cdot\rangle$ to denote the Euclidean norm and Euclidean inner product, respectively. We use $\|\cdot\|_{\rm HS}$ to denote the Hilbert-Schmidt norm of matrices. Write $\mathbb{R}^{+}=[0, \infty)$ and $\mathbb{R}^{-}=(-\infty, 0]$. For $a, b \in \mathbb{R}$, 
	\[
	a \wedge b=\min \{a, b\}\text{ and } a \vee b=\max \{a, b\}.
	\]
	Let $A$ be a vector or a matrix, we use $A^{\top}$ to denote its transpose. Let $\vec{0}=(0,\cdots,0)^{\top}$ and $\vec{1}=(1,\cdots,1)^{\top}$.  Suppose that all random objects are defined on a complete filtrated probability space $\left(\Omega, \mathscr{F}_{,}\left\{\mathscr{F}_t\right\}_{t \geq 0}, \mbb{P}\right)$.  
	
	To make our argument in the following more precise, we give an explicit construction of the complete filtrated probability space. Let
	\[
	\Omega_{1}=\left\{ \omega:[0, \infty) \rightarrow \mathbb{R}^{d} ;\omega \text { is continuous with } \omega(0)=0\right\},
	\]
	which is endowed with the locally uniform convergence topology and the Wiener measure ${\mathbb{P}_{1}}$ so that the coordinate process ${W(t, \omega):=\omega(t), t \geq 0}$, is a standard ${m}$-dimensional Brownian motion. Put
	\[
	{\Omega_{2}=\{\omega:[0, \infty) \rightarrow \mbb S; \omega \text{ is right continuous with left limit}\}}
	\] 
	endowed with Skorokhod topology and a probability measure ${\mbb P_{2}}$ so that the coordinate process
	$
	\Lambda(t, \omega):=\omega(t),  t \geq 0,
	$
	is a continuous time Markov chain with the generator ${Q=(q_{ij})_{N\times N}}$. Let
	\[
	(\Omega, \mathcal{F}, \left\{\mathscr{F}_t\right\}_{t \geq 0},\mathbb{P})=\left(\Omega_{1} \times \Omega_{2},  \mathscr{B}\left(\Omega_{1}\right) \otimes \mathscr{B}\left(\Omega_{2}\right), \left\{\mathscr{F}_t^{(1)}\otimes\mathscr{F}_t^{(2)}\right\}_{t \geq 0}, \mathbb{P}_{1} \times \mathbb{P}_{2}\right),
	\]
	where $\{\mathscr{F}_t^{(1)}\}_{t\geq 0}$ and $\{\mathscr{F}_t^{(2)}\}_{t\geq 0}$ are the natural filtration generated by $(W(t))$ and $(\Lambda(t))$ , respectively.	Hence, under ${\mathbb{P}=\mathbb{P}_{1} \times \mathbb{P}_{2}}$, for ${\omega=\left(\omega_{1}, \omega_{2}\right), t \mapsto \omega_{1}(t)}$ is a Wiener process, and ${t \mapsto \omega_{2}(t)}$ is a continuous-time Markov chain with the generator ${Q=(q_{ij})_{N\times N}}$. Moreover, the Wiener process is independent of the continuous-time Markov chain. Throughout this work, we will work on this complete filtrated probability space constructed above.
	
	Set $E:=C_r\times \mbb S$. Denote by $\mathscr{B}_b(E)$ (resp., $\mathscr{C}_b(E)$) the set of all bounded Borel functions (resp., continuous and bounded functions) in $E$. Let $\mathscr{P}\left(\mathbb{R}^{-}\right)$ and $\mathscr{P}\left(E\right)$ denote the set of probability measures in $\mathbb{R}^{-}$ and $E$, respectively. In order to overcome the difficulties caused by infinite delay, we further define $\mathscr{P}_c\left(\mathbb{R}^{-}\right)$ for any $c>0$ as shown below:
	$$
	\mathscr{P}_c\left(\mathbb{R}^{-}\right)=\left\{\rho \in \mathscr{P}\left(\mathbb{R}^{-}\right) ; \delta_c(\rho):=\int_{-\infty}^0 \mathrm{e}^{-c \theta} \rho(\mathrm{d} \theta)<\infty\right\} .
	$$
	
	To guarantee the existence and uniqueness of the solution map, we first assume that for any $k \in \mbb S$, $b(\cdot,k)$ and $\sigma(\cdot,k)$ are measurable and locally Lipschitz throughout this paper. Moreover, we impose the following assumptions $(\mathbf{A})$.
	
	\noindent(\hypertarget{A0}{A0}) Let ${G(\textbf{0},l)\equiv \vec{0}}$.   Given  $p\geq 2$, there exist a probability measure $\rho \in \mathscr{P}_{p r}\left(\mathbb{R}^{-}\right)$ and a  constant ${\kappa >0}$  with $\kappa_1=\kappa\delta_{p r}^{\frac{1}{p}}(\rho)\in (0,1)$ such that\[
	|G(\varphi,l)-G(\psi,l)|^p \leq \kappa^p \int_{-\infty}^0|\varphi(\theta)-\psi(\theta)|^p\rho(\d \theta)
	\]
   for $\varphi$, $\psi \in C_r$ and $l \in \mbb S$.
	
	\noindent(\hypertarget{A1}{A1}) For any $l \in \mbb S$, there exist constants $\alpha(l) \in \mathbb{R}$ and $\beta(l) \in \mathbb{R}^{+}$ such that
	$$
	\begin{aligned}
		&\langle\varphi(0)-\psi(0)+G(\psi,l)-G(\varphi,l), b(\varphi, l)-b(\psi, l)\rangle\\
		&\quad\leq  \alpha(l)|\varphi(0)-\psi(0)+G(\psi,l)-G(\varphi,l)|^2+\beta(l) \int_{-\infty}^0|\varphi(\theta)-\psi(\theta)|^2 \rho(\mathrm{d} \theta)
	\end{aligned}
	$$
	for any  $\varphi$, $\psi \in   C_r$.
	
	\noindent(\hypertarget{A2}{A2}) There exists a  constant $\gamma>0$ such that
	$$
	\|\sigma(\varphi, l)-\sigma(\psi, l)\|_{\text {HS }}^2 \leq \gamma\int_{-\infty}^0|\varphi(\theta)-\psi(\theta)|^2 \rho(\mathrm{d} \theta) 
	$$
 for any $\varphi$, $\psi \in  C_r$ and  $l \in \mbb S$.
	
	\begin{remark}\label{eq:toutengg}
		Noting that $G(\textbf{0},l)\equiv \vec{0}$, we have 
		\begin{equation*}
			|G(\varphi,l)|^p\leq \kappa^p \int_{-\infty}^0|\varphi(\theta)|^p\rho(\d \theta).
		\end{equation*}
		Moreover,  (\hyperlink{A0}{A0}) implies that
		\begin{equation}\label{eq:xiangai}
			|G(\varphi,l)-G(\psi,l)|^p\leq\kappa^p\int_{-\infty}^{0} \e^{-p r \theta} \e^{p r \theta}|\varphi(\theta)-\psi(\theta)|^{p} \rho(d \theta) \leq \kappa_1^p\|\varphi-\psi\|_{r}^{p}.
		\end{equation}
	\end{remark}
	When $\S$ is an infinite set, we need some additional assumptions $(\mathbf{H})$ as follows.
	
	\noindent(\hypertarget{(H1)}{H1}) Suppose that for any $k\in \S$,
	\begin{equation*}
		q_k:=-q_{kk}=\sum_{l\neq k}q_{kl}\leq M,
	\end{equation*}
	where $M$ is a positive constant independent of $k$.

	\noindent(\hypertarget{H2}{H2}) Suppose that  ${\alpha_{sup}}:=\sup_{k\in \S}\alpha(k)$, ${\beta_{sup}}:=\sup_{k\in \S}\beta(k)<\infty$, $\sup_{k\in \S}|b(\textbf{0}, k)|<\infty$ and $\sup_{k\in \S}\|\sigma(\textbf{0}, k)\|_{\rm H S}<\infty$.

	We also need the following lemma.
	\begin{lemma}\cite[Lemma 6.4.1]{Maostochatic}
		Let $p>1$, $\varepsilon>0$,  and $a,b\in\mathbb{R}$. Then
		\begin{equation}\label{eq:jianbang}
			|a+b|^p\leq \left(1+\varepsilon^{\frac{1}{p-1}}\right)^{p-1}\left(|a|^p+\frac{|b|^p}{\varepsilon}\right).
		\end{equation}	
	\end{lemma}
	\begin{remark}
		Letting $\varepsilon=\left(\frac{\varepsilon^{\prime}}{1-\varepsilon^{\prime}}\right)^{p-1}$, $\varepsilon^{\prime}\in (0,1)$, yields that
		\begin{equation}\label{eq:yao}
			|a+b|^p\leq \frac{1}{(1-\varepsilon^{\prime})^{p-1}}|a|^p+\frac{1}{(\varepsilon^{\prime})^{p-1}}|b|^p.
		\end{equation}
	\end{remark}
	
	In what follows, we write ${X(t)}$, $X_t$ and ${\Lambda(t)}$ as ${X^{\xi, i}(t)}$, ${X_{t}^{\xi, i}}$ and ${\Lambda^{i}(t)}$ respectively, to emphasize the initial data $(X_{0},\Lambda(0))=(\xi,i)$.
	\section{Existence and uniqueness}\label{sec:3}
	In this section, we give the existence and uniqueness of the solution map to the RNSFDE $\left(X_t, \Lambda(t)\right)$ defined by \eqref{eq:fangcheng1} and \eqref{eq:fangcheng2}.  
	\begin{definition}
		A $C_r$-valued stochastic process ${(X_t)_{ t\geq 0}}$ is called a solution map to the RNSFDE if
		\begin{enumerate}[(1)]
			\item $X_t$ is continuous and ${\mathcal{F}_{t}}$-adapted;
			\item ${\int_{{0}}^{\infty}\|X_t\|_r^{2} \d t<\infty}$, almost surely;
			\item ${X_0=\xi}$ and with probability 1,
			\begin{equation}\label{eq:NSFDE2}
				X(t)-G\left(X_{t},\Lambda(t)\right)=\xi(0)-G\left(\xi,i\right)+\int_{{0}}^{t} b\left(X_{s},\Lambda(s)\right) \d s+\int_{{0}}^{t} \sigma\left(X_{s},\Lambda(s)\right) \d W(s),
			\end{equation}
		\end{enumerate}
		for any  $t \geq 0$. A solution map ${(X_t)}$ is said to be unique if any other solution map $(Y_t)$ is indistinguishable from it, that is,
		\[
		\mbb P\left\{\|X_t-Y_t\|_r=0, \text{ for all }t \geq 0\right\}=1 .
		\]
		Moreover, the corresponding $\mbb R^n$-valued stochastic process ${(X(t))_{ t\geq 0}}$ is called a solution of the RNSFDE.
	\end{definition}
	To proceed, we need to introduce a family of stochastic processes. For each $k\in \mbb S$, consider the following neutral stochastic functional differential equation, i.e.
	\begin{equation}\label{eq:NSFGE}
		\d\{X^{(k)}(t)-G(X_t^{(k)},k)\}=b(X_t^{(k)},k)\d t+\sigma(X_t^{(k)},k)\d W(t).
	\end{equation}
	Then, we introduce the definition of the local solution map with respect to Eq. \eqref{eq:NSFGE} and establish the well-posedness for Eq. \eqref{eq:NSFGE}.
	\begin{definition}\label{def:shengtong}
		Let $\tau_e$ be a stopping time such that $0\leq \tau_e\leq \infty$. A continuous $C_r$-valued and \( \mathcal{F}_{t}^{(1)} \)-adapted process \( X_t^{(k)}\), \(0\leq t \leq \tau_{e} \), is called a local solution map of \eqref{eq:NSFGE} with the initial data \( \xi \in C_{r} \), if \( X_0=\xi \) and for any \( t \geq 0 \)  and  \( n \geq 1 \),
		\[
		X^{(k)}(t\wedge \tau_{n})-G\left(X_{t\wedge \tau_{n}}^{(k)},k\right)=\xi(0)-G\left(\xi,k\right)+\int_{0}^{t\wedge\tau_{n}} b\left(X_{s}^{(k)},k\right) \d s+\int_{0}^{t\wedge\tau_{n}} \sigma\left(X_{s}^{(k)},k\right) \d W(s) 
		\]
		holds with probability 1, where \( \left\{\tau_{n}\right\}_{n \geq 1} \) is a nondecreasing sequence of stopping times such that \( \tau_{n} \rightarrow \tau_{e} \) a.s., as \( n \rightarrow \infty \). If, furthermore, \( \lim \sup _{t \rightarrow \tau_{e}}\|X^{(k)}_t\|_r=\infty \) is satisfied a.s., whenever \( \tau_{e}<\infty \) a.s., then it is called a maximal local solution map and \( \tau_{e} \) is called the explosion time. A maximal local  solution map  \( X_t^{(k)},0\leq t<\tau_{e} \),  is called a global solution map when \( \tau_{e}=\infty \). A maximal local  solution map  \( X_t^{(k)},0\leq t<\tau_{e}, \) is said to be unique if for any other maximal local  solution map \( \bar{X}_t^{(k)}\), \(0\leq t<\bar{\tau}_{e} \), we have \( \tau_{e}=\bar{\tau}_{e} \) and \( X_t^{(k)}=\bar{X}_t^{(k)} \) for \( 0\leq t<\tau_{e} \) a.s..
	\end{definition}
	\begin{lemma}\label{theorem:chabei}
		For fixed $k\in \S$, suppose that assumptions (\hyperlink{A0}{A0})-(\hyperlink{A2}{A2}) hold for $p=2$,  then for any initial data $\xi\in C_r$,  Eq. \eqref{eq:NSFGE} has a global solution map $X_t^{(k)}$ almost surely,  which is continuous and $\mcl F_t^{(1)}$-adapted.
	\end{lemma}
	\begin{proof}
		For any $n\geq 1$, let us define truncation the functional \( b_{n} \) and \( \sigma_{n} \) as follows:
		\[
		\begin{array}{c}
			b_{n}(\varphi,k)=\left\{\begin{array}{ll}
				b(\varphi,k), & \|\varphi\|_{r} \leq n, \\
				b\left(n \varphi/ \|\varphi\|_{r},k\right), & \|\varphi\|_{r}>n,
			\end{array}\right. 
			\sigma_{n}(\varphi,k)=\left\{\begin{array}{ll}
				\sigma(\varphi,k), & \|\varphi\|_{r} \leq n, \\
				\sigma\left(n \varphi/ \|\varphi\|_{r},k\right), & \|\varphi\|_{r}>n.
			\end{array}\right.
		\end{array}
		\]
		It is obvious that  \( b_{n} \) and \( \sigma_{n} \) satisfy the global Lipschitz condition and the linear growth condition. Consider
		\begin{equation}\label{eq:nicai}
			\d\{X^{(k)(n)}(t)-G(X_t^{(k)(n)},k)\}=b_n(X_t^{(k)(n)},k)\d t+\sigma_n(X_t^{(k)(n)},k)\d W(t),\quad t\geq 0
		\end{equation}
		with the initial data \( \xi \in C_r \). Moreover, it follows from \eqref{eq:xiangai} that $G(\cdot,k)$ satisfies the global Lipschitz condition. Hence by Theorem \cite[Theorem A.4]{zhang2024stability}, there is a unique continuous solution map $X^{(k)(n)}_t$ satisfying  \eqref{eq:nicai}.   Define a stopping time sequence
		\[
		\tau_{n}=\inf \left\{t \in[0,\tau_e) ,\|X_{t}^{(k)(n)}\|_{r}\geq n\right\},
		\]
		with the usual convention \( \inf \emptyset=\infty \). 	It is not difficult to show that
		\[
		X_{t}^{(k)(n+1)}=X_{t}^{(k)(n)},\quad \text{if }0\leq t\leq \tau_n.
		\]
		This implies that	\(\left\{\tau_{n}\right\}_{n \geq 1} \) is a nondecreasing sequence and \( \tau_{n} \rightarrow \tau_{\infty} \leq \) \( \tau_{e} \) a.s. as \( n \rightarrow \infty \).	Define \( X_t^{(k)}\), \(0\leq t \leq \tau_{e} \), by
		\[
		X_t^{(k)}=X_t^{(k)(n)},\quad t\in [\tau_{n-1},\tau_n),\;n\geq 1.
		\]
		Note that for any $ t \in\left[0, \tau_{n}\right]$, $X^{(k)}_t=X^{(k)(n)}_t $. Letting \( n \rightarrow \infty \), it follows that for any \( t \in\left[0, \tau_{e}\right) \), there exists a  local solution map \( X^{(k)}_t \) for \eqref{eq:NSFGE}. By a standard procedure, we can show \( X_t^{(k)}\), \(0\leq t \leq \tau_{e} \), is also the unique maximal local solution map.
		
		To show that this solution map is global, it is sufficient to prove that \( \tau_{e}=\infty \) a.s.  If we can show \( \tau_{\infty}=\infty \) a.s, then \( \tau_{e}=\infty \) a.s. This is equivalent to proving that for any \( T>0\), \(\mathbb{P}\left(\tau_{n} \leq T\right) \rightarrow 0 \) as \( n \rightarrow \infty \). 
		
		By the definitions of $\tau_n$ and $X_t^{(k)}$, $\|X_{\tau_n}^{(k)}\|_{r} = n$, which implies 
		\begin{equation}\label{eq:xiawu}
			\begin{aligned}
				n^2\mbb P(\tau_n\leq T)&=\E\left(\|X_{\tau_{n}}^{(k)}\|_r^2 \mathbf{1}_{\{\tau_n\leq T \}} \right)\leq \E\|X_{T\wedge\tau_n}^{(k)}\|_r^2 \\
				&\leq \|\xi\|_r^2+\E\left(\sup _{0< t\leq T\wedge\tau_n} \e^{2 r t}\left|X^{(k)}(t)\right|^{2}\right).
			\end{aligned}
		\end{equation}
		Set $H(t):=\E\left(\sup _{0< s\leq t\wedge\tau_n} \e^{2 r s}\left|X^{(k)}(s)\right|^{2}\right)$, $0\leq t\leq T$.  In what follows, we shall estimate $H(t)$.  Letting $p=2$  and $\varepsilon=\kappa_1=\kappa\delta_{2 r}^{\frac{1}{2}}(\rho)\in (0,1)$ in  \eqref{eq:yao}, we have
		\begin{equation}\label{eq:xielun}
			\begin{aligned}
				H(t)&\leq \frac{1}{1-\kappa_1} \mathbb{E}\left(\sup_{0< s \leq t\wedge\tau_n} \e^{2 r s}\left|X^{(k)}(s)-G\left(X_{s}^{(k)},k\right)\right|^{2}\right)
				+\frac{1}{\kappa_1} \mathbb{E}\left(\sup_{0< s \leq t\wedge\tau_n} \e^{2 r s}\left|G\left(X_{s}^{(k)},k\right)\right|^{2}\right). 
			\end{aligned}
		\end{equation}
		By (\hyperlink{A0}{A0}), we have
		\begin{equation*}
			\begin{aligned}
				\mathbb{E}\left(\sup_{0< s \leq t\wedge\tau_n} \e^{2 r s}\left|G\left(X_{s}^{(k)},k\right)\right|^{2}\right)
				&\leq  \kappa^2 \mathbb{E}\left(\sup _{0<s \leq t\wedge\tau_n}\int_{-\infty}^0 \e^{-2r\theta}\e^{ 2r(s+\theta)}\left|X^{(k)}(s+\theta)\right|^{2}\rho(\d \theta)\right) \\
				&\leq \kappa^2\delta_{2r}(\rho) \mathbb{E}\left(\sup_{-\infty< s\leq t\wedge\tau_n} \e^{2 r s}\left|X^{(k)}(s)\right|^{2}\right) \\
				&\leq 
				\kappa_1^2\left(\|\xi\|_{r}^{2}+ \mathbb{E}\Big(\sup _{0< s \leq t\wedge\tau_n} \e^{2 r s}\left|X^{(k)}(s)\right|^{2}\Big)  \right), \\
			\end{aligned}
		\end{equation*}
		which, together with \eqref{eq:xielun}, implies that 
		\begin{equation}\label{eq:meichuan}
			H(t)\leq\frac{\kappa_1}{1-\kappa_1}\|\xi\|_{r}^{2}+\frac{1}{(1-\kappa_1)^{2}} \E\left(\sup_{0< s \leq t\wedge\tau_n} \e^{2 r s}\left|X^{(k)}(s)-G\left(X_{s}^{(k)},k\right)\right|^{2}\right).
		\end{equation}
		It follows from (\hyperlink{A0}{A0})-(\hyperlink{A2}{A2}) that
		\[
		\begin{aligned}
			2\langle&\varphi(0)-G(\varphi,k), b(\varphi, k)\rangle+\|\sigma(\varphi,k)\|_{\rm HS}^2\\
			&=2\langle\varphi(0)-\vec{0}+G(\mathbf{0},k)-G(\varphi,k), b(\varphi, k)-b(\textbf{0},k)\rangle+2\langle\varphi(0)-G(\varphi,k),b(\mathbf{0},k)\rangle\\
			&\quad+\|\sigma(\varphi,k)-\sigma(\mathbf{0},k)+\sigma(\mathbf{0},k)\|_{\rm HS}^2\\
			&\leq  C+(2\alpha(k)+1)|\varphi(0)-G(\varphi,k)|^2+2(\beta(k) +\gamma)\int_{-\infty}^0|\varphi(\theta)|^2 \rho(\mathrm{d} \theta) \\
			&\leq C+2(2\alpha(k)+1)|\varphi(0)|^2+2\Big(\beta(k) +\gamma+\kappa^2(2\alpha(k)+1)\Big)\int_{-\infty}^0|\varphi(\theta)|^2 \rho(\mathrm{d} \theta) .
		\end{aligned}
		\]
		Thus, applying Itô formula to \( \e^{2 r t}\left|X^{(k)}(t)-G\left(X_{t}^{(k)},k\right)\right|^{2}\) yields for any \( t \geq 0 \),
		\begin{equation}\label{eq:shuijiao}
			\begin{aligned}
				\E&\left(\sup _{0< s\leq t\wedge\tau_n} \e^{2 r s}\left|X^{(k)}(s)-G\left(X_{s}^{(k)},k\right)\right|^{2}\right)\\
				&\leq |\xi(0)-G(\xi,k)|^2+\E\Bigg(\sup _{0< s\leq t\wedge\tau_n}\int_{{0}}^s\e^{2ru}\Big(2r|X^{(k)}(u)-G(X_u^{(k)},k)|^2\\
				&\quad+2\langle X^{(k)}(u)-G(X_u^{(k)},k),b(X_u^{(k)},k) \rangle+\|\sigma(X_u^{(k)},k)\|_{\text{HS}}^2\Big)\d u\Bigg)\\
				&\quad+ 2\E\left(\sup _{0< s\leq t\wedge\tau_n}\int_{{0}}^s\e^{2ru}\langle X^{(k)}(u)-G(X_u^{(k)},k),\sigma(X_u^{(k)},k)\d W(u) \rangle \right)\\
				&\leq |\xi(0)-G(\xi,k)|^2+C\e^{2rt}+ C\E\int_{{0}}^{t\wedge\tau_n} \e^{2rs} |X^{(k)}(s)|^2\d s\\
				&\quad+C\E \int_{{0}}^{t\wedge\tau_n} \int_{-\infty}^{{0}}\e^{2rs}|X^{(k)}(s+\theta)|^2 \rho(\d \theta)\d s\\
				&\quad+ 2\E\left(\sup _{0< s\leq t\wedge\tau_n}\int_{{0}}^s\e^{2ru}\langle X^{(k)}(u)-G(X_u^{(k)},k),\sigma(X_u^{(k)},k)\d W(u) \rangle \right).
			\end{aligned}
		\end{equation}
		By means of the BDG inequality, it follows that
		\begin{equation*}
			\begin{aligned}
				2\E&\left(\sup _{0< s\leq t\wedge\tau_n}\int_{{0}}^s\e^{2ru}\langle X^{(k)}(u)-G(X_u^{(k)},k),\sigma(X_u^{(k)},k)\d W(u) \rangle \right)\\
				&\leq 8\sqrt{2}\E\left(\int_{{0}}^{t\wedge\tau_n}\e^{4rs}|\langle X^{(k)}(s)-G(X_s^{(k)},k),\sigma(X_s^{(k)},k) \rangle|^2\d s  \right)^{\frac{1}{2}}\\
				&\leq 8\sqrt{2}\E\left(\sup_{{0} < s \leq {t\wedge\tau_n}} \e^{2rs}|X^{(k)}(s)-G(X_s^{(k)},k)|^2\int_{{0}}^{t\wedge\tau_n}\e^{2rs}\|\sigma(X_s^{(k)},k)\|_{\text{HS}}^2\d s \right)^{\frac{1}{2}}\\
				&\leq \frac{1}{2}\E\left(\sup_{{0} < s \leq {{t\wedge\tau_n}}} \e^{2rs}|X^{(k)}(s)-G(X_s^{(k)},k)|^2 \right)+64\E\int_{{0}}^{t\wedge\tau_n}\e^{2rs}\|\sigma(X_s^{(k)},k)\|_{\text{HS}}^2\d s\\
				&\leq \frac{1}{2}\E\left(\sup_{{0} < s \leq {t\wedge\tau_n}} \e^{2rs}|X^{(k)}(s)-G(X_s^{(k)},k)|^2 \right)+64\gamma\E\int_{0}^{{t\wedge\tau_n}} \int_{-\infty}^{0} \e^{2r s}|X(s+\theta)|^{2} \rho(\mathrm{d} \theta) \mathrm{d} s+C\e^{2r t}.
			\end{aligned}
		\end{equation*}
		Plugging this into \eqref{eq:shuijiao} implies 
		\begin{equation}\label{eq:fuzhi}
			\begin{aligned}
				\E&\left(\sup _{0< s\leq t\wedge\tau_n} \e^{2 r s}\left|X^{(k)}(s)-G\left(X_{s}^{(k)},k\right)\right|^{2}\right)\\
				&\leq 2|\xi(0)-G(\xi,k)|^2+C\e^{2rt}+ C\E\int_{{0}}^{t\wedge\tau_n} \e^{2rs} |X^{(k)}(s)|^2\d s\\
				&\quad+C\E \int_{{0}}^{t\wedge\tau_n} \int_{\infty}^{{0}}\e^{2rs}|X^{(k)}(s+\theta)|^2 \rho(\d \theta)\d s\\
			\end{aligned}
		\end{equation}
		On the other hand, by the Fubini theorem and the variable substitution technique,
		we deduce that
		\begin{equation}\label{eq:zhuyaoanother}
			\begin{aligned}
				\int_{0}^{t\wedge\tau_n} &\int_{-\infty}^{0} \e^{2r s}|X^{(k)}(s+\theta)|^{2} \rho(\mathrm{d} \theta) \mathrm{d} s \\
				&=\int_{0}^{t} \int_{-\infty}^{-s} \e^{-2 r\theta} \e^{2 r(s+\theta)}|X^{(k)}(s+\theta)|^{2} \rho(\mathrm{d} \theta) \mathrm{d} s+\int_{-t\wedge\tau_n}^{0}\int_{0}^{t\wedge\tau_n+\theta} \e^{2r(s-\theta)}|X^{(k)}(s)|^{2} \mathrm{d} s \rho(\mathrm{d} \theta) \\
				&\leq \delta_{2r}\left(\rho\right)\|\xi\|_{r}^{2}t+\delta_{2r}\left(\rho\right) \int_{0}^{t\wedge\tau_n}\e^{2r s}|X^{(k)}(s)|^2\mathrm{d}s.
			\end{aligned}
		\end{equation}
		Substituting  \eqref{eq:zhuyaoanother} into \eqref{eq:fuzhi}, we conclude that
		\[
		\E\left(\sup _{0< s\leq t\wedge \tau_n} \e^{2 r s}|X^{(k)}(s)-G\left(X_{t},k\right)|^{2}\right)\leq C+C\E\int_{{0}}^{t\wedge\tau_n} \e^{2rs}|X(s)|^2\d s,
		\]
		which, together with \eqref{eq:meichuan}, implies that
		\[
		H(t)\leq C+C\int_{{0}}^t H(s)\d s.
		\]
		By the Gronwall inequality, we obtain $H(t)\leq C\e^{Ct}.$
		Choosing $t=T$ gives 
		\begin{equation}\label{Eq:1031:1}
			n^2\mbb P(\tau_n\leq T)\leq \|\xi\|_r^2+C\e^{CT}.
		\end{equation}
		Note that the right side is independent of $n$ in \eqref{Eq:1031:1}. Letting $n\to \infty$ yields 
		\[
		\limsup_{n \rightarrow \infty}\mbb P(\tau_n\leq T) = 0,
		\]
		which implies that \eqref{eq:NSFGE} has a unique global solution map $X_t^{(k)}$ on $[0, \infty)$ almost surely.
	\end{proof}
	Before establishing the existence and uniqueness of the solution map to the RNSFDE $\left(X_t, \Lambda(t)\right)$, we first introduce Skorokhod’s representation of the Markov chain $(\Lambda(t))$ in terms of the Poisson random measure as in \cite{nguyen2016modeling}. Precisely, let
	\[
	\Delta_{1 2}=\left[0, q_{1 2}\right), \Delta_{1 l}=\left[\sum_{j=2}^{l-1}q_{1j},\sum_{j=2}^{l}q_{1j}\right),\quad l\geq 3,
	\]
	and for each $k \in \mathbb{S}$ and $k>1$, let
	\[
	\Delta_{k 1}=\left[0, q_{k 1}\right), \Delta_{k l}=\left[\sum_{j=1,j\neq k}^{l-1}q_{kj},\sum_{j=1,j\neq k}^{l}q_{kj}\right),\quad l>1,l\neq k.
	\]
	Let
	\[
	U_{k}=\bigcup_{l \geq 1, l \neq k} \Delta_{kl}, \quad k \geq 1.
	\]
	For notation convenience, we put \( \Delta_{kk}=\varnothing \) and \( \Delta_{kl}=\varnothing \) if \( q_{kl}=0, k, l \in \mathbb{S}\). Note that for each $k \in \mathbb{S}$, $\{\Delta_{kl}:l\in \mathbb{S}\}$ are disjoint intervals, and the length of the interval $\Delta_{kl}$ is equal to $q_{kl}$.  When $\S$ is a finite set, it is obvious that  $\mathfrak{m}\left(U_{k}\right)$ is bounded  by $\max_{1\leq l\leq N}q_{l}$  for any $k\in\S$,  where \( \mathfrak{m}(\mathrm{d} x) \) denotes the Lebesgue measure over \( \mathbb{R} \). When $\S$ is an  infinite set,  (\hyperlink{H1}{H1}) ensures that $\mathfrak{m}\left(U_{k}\right)$ is also bounded. Without loss generality, we assume that $\mathfrak{m}\left(U_{k}\right)$ is bounded above by $M$, whether $\mathbb{S}$ is a finite set or not, 
	
	Let \( \xi_{n}^{(k)}, k, n=1,2, \ldots \), be \( U_{k} \)-valued random variables with
	$\mathbb{P}(\xi_{n}^{(k)} \in \mathrm{d} x)=\mathfrak{m}(\d x)/\mathfrak{m}(U_{k}),
	$ and \( \tau_{n}^{(k)}, k, n \geq 1 \), be non-negative random variables satisfying \( \mathbb{P}(\tau_{n}^{(k)}>t)=\exp\{-\mathfrak{m}\left(U_{k}\right)t\} \), \( t \geq 0 \). Suppose that \( \{\xi_{n}^{(k)}, \tau_{n}^{(k)}\}_{k, n \geq 1} \) are all mutually independent. Put
	\[
	\zeta_{0}^{(k)}=0, \quad k \geq 1;\quad \zeta_{n}^{(k)}=\tau_{1}^{(k)}+\cdots+\tau_{n}^{(k)}, \quad k,n  \geq 1.
	\]
	Let
	\[
	D_{p}=\bigcup_{k \geq 1} \bigcup_{n \geq 0}\left\{\zeta_{n}^{(k)}\right\}\text{ and }p(\zeta_{n}^{(k)})=\xi_{n}^{(k)}\quad \text { for } k,n  \geq 1.
	\]
	Correspondingly, put
	\begin{equation}\label{eq:bujini}
		N([0, t] \times A)=\#\left\{ s \in D_{p}:0< s \leq t, p(s) \in A\right\}, t>0, A \in \mathscr{B}([0, \infty)) .
	\end{equation}
	As a consequence, we get a Poisson point process \( (p(t))_{t\geq 0} \) and a Poisson random measure \( {N}(\mathrm{d} t, \mathrm{d} u) \) with intensity \( \mathfrak{m}(\mathrm{d} u)\mathrm{d} t  \). Moreover, we know that \( \mathfrak{m}(\d u) \d t\) is the compensator of Poisson random measure $N(\d t, \d u)$. Define a function $h: \mathbb{S} \times\left[0, M\right]$ by
	\[
	h(k, u)=\sum_{l \in \mathbb{S}}(l-k) \textbf{1}_{\triangle_{k l}}(u).
	\]
	Then, $(\Lambda(t))$ can be reformulated as
	\begin{equation}\label{Eq:1026:1}
		\begin{split}
			\d \Lambda(t) & =\int_{[0, M]} h(\Lambda(t-), u) N(\d t, \d u). \\
		\end{split}
	\end{equation}
	With the aid of Lemma \ref{theorem:chabei} and \eqref{Eq:1026:1}, we can show the existence and uniqueness of the solution map to the RNSFDE by a similar argument as in \cite[Theorem 3.1]{nguyen2016modeling}. For simplicity, we omit the details.
	\begin{theorem}\label{Eq:1026:2}
		Suppose that assumptions (\hyperlink{A0}{A0})-(\hyperlink{A2}{A2}) hold for $p=2$. When   
		$\S$ is an infinite set, we need to additionally assume that  (\hyperlink{B1}{B1}) holds. Then for any initial data $(\xi,i)\in \mathscr{C}_r\times \S$, the RNSFDE admits a unique  solution map. 
	\end{theorem}

	\section{Markovian switching in a finite state space}\label{sec:4}
	In this section, we consider the exponential ergodicity for the RNSFDE $\left(X_t, \Lambda(t)\right)$ with $N<\infty$.   To this aim, we first need to introduce an auxiliary process $Y^{\xi, i}(t):=X^{\xi, i}(t)-G(X_t^{\xi, i},\Lambda^{i}(t))$.  In what follows, 
	let's briefly introduce our idea of the proof of Theorem \ref{eq:boziteng} . First, we establish the uniform boundedness of the auxiliary process $(Y(t))$ and the convergence for two processes $(Y^{\xi, i}(t))$ and $(Y^{\eta, i}(t))$ from different initial data in the sense of $L^p(\Omega;\mathbb{R}^d)$. Then, we further generalize  these results to the segment process $(X_t)$. Finally, we obtain the exponential  ergodicity  under Wasserstein distance based on the above results.

	Before we proceed, we give some warm-up notions. For any $k\in\S$, set
	\begin{equation}\label{eq:07gggh}
		f(k):=(\bar{\beta}+\gamma_p)\left(p-2+\frac{2\delta_{pr-p\underline{\alpha}}(\rho)}{(1-\kappa_2)^{p}}\right)+p\alpha(k),
	\end{equation}
	where $\gamma_p=(p-1)/2$, $\kappa_2:=\kappa\delta_{pr-p\underline{\alpha}}^{\frac{1}{p}}(\rho) $,  $\underline{\alpha}:=min_{i\in\S}\alpha(i)$, \( \alpha(k),\gamma \), and \( \rho \) are introduced in (\hyperlink{A0}{A0})-(\hyperlink{A2}{A2}). Moreover, we introduce the quantities associated to \( f(k)\), \( k \in \S \). Set
	\[
	\widehat{Q}:=Q+\operatorname{diag}(f(1), f(2), \ldots, f(N)), \quad \zeta:=-\max _{s \in \operatorname{spec}(\widehat{Q})} \operatorname{Re}(s),
	\]
	where \( \operatorname{diag}(f(1), f(2), \ldots, f(N)) \) denotes the diagonal matrix generated by  \( (f(1), f(2), \ldots, f(N))^{\top};\)   \(\operatorname{spec}(\widehat{Q}) \) is the spectrum of \( \widehat{Q} \) and \( \operatorname{Re}(\lambda) \) is the real part of \( \lambda \). 
	
	By virtue of \cite[Proposition 4.1]{bardet2010long}, we now state a useful lemma which relates \( \zeta \) with the exponential functional of \( f(\Lambda(t)) \) along the trajectories of \( \Lambda(t) \).
	\begin{lemma}\label{eq:wuwei}
		There exist constants \( 0<c_{1}<c_{2}< \) \( \infty \) such that for any \( i \in \S \) and \( 0 \leq  u<t \),
		\begin{equation}\label{Eq:1210}
			c_{1} \mathrm{e}^{-\zeta(t-u)} \leq \mathbb{E}_{\mathbb{P}_2}\left[\exp \left(\int_{u}^{t} f\left(\Lambda^{i}(v)\right) \mathrm{d} v\right)\right] \leq c_{2} \mathrm{e}^{-\zeta(t-u)},
		\end{equation}
		where $\mathbb{E}_{\mathbb{P}_2}$ denotes the expectation with respect to $\mathbb{P}_2$.
	\end{lemma}
	
	\subsection{Boundedness and Convergence  of $Y(t)$}
	\begin{lemma}\label{thm:opao1}
		Let $p\geq 2$ and  assumptions (\hyperlink{A0}{A0})-(\hyperlink{A2}{A2})  hold. Suppose further
		that  \( \underline{\alpha}<0 \) , \( \zeta>0 \)  and $\kappa_2 \in (0,1)$. Then there exists a constant $C>0$  such that for any   $(\xi,i)\in E$ and  $t\geq 0 $,
		\begin{equation}\label{eq:aaa}
			\mbb{E}\left|Y^{\xi, i}(t)\right|^{p} \leq C(1+ \|\xi\|_{r}^{p}).
		\end{equation}      
	\end{lemma}
	\begin{proof}
		First, for any $\varepsilon_1,\varepsilon_2>0$, it follows from  (\hyperlink{A1}{A1}) and (\hyperlink{A2}{A2}) that for all $(\varphi,k) \in  E$, 
		\[
		\begin{aligned}
			\langle\varphi(0) & -G(\varphi, k), b(\varphi, k)\rangle+\frac{p-1}{2}\|\sigma(\varphi, k)\|_{\rm H S}^{2} \\
			&\leq  \langle\varphi(0)-0+G(\textbf{0}, k)-G(\varphi, k), b(\varphi, k)-b(\textbf{0}, k)\rangle+\langle\varphi(0)-G(\varphi, k), b(\textbf{0}, k)\rangle \\
			& \quad+\frac{p-1}{2} \| \sigma ( \varphi, k)-\sigma(\textbf{0}, k)+\sigma(\textbf{0}, k) \|_{\rm H S}^{2} \\
			& \leq  \alpha(k)|\varphi(0)-G(\varphi, k)|^{2}+\beta(k) \int_{-\infty}^{0}|\varphi(\theta)|^{2}\rho(\d \theta)+\frac{\varepsilon}{2}|\varphi(0)-G( \varphi, k)|^{2}+\frac{1}{2 \varepsilon_{1}}|b(\textbf{0}, k)|^{2} \\
			& \quad+\frac{p-1}{2}\left[\left(1+\varepsilon_{2}\right)\|\sigma(\varphi, k)-\sigma(\textbf{0}, k)\|_{\rm H S}^{2}+\left(1+\frac{1}{\varepsilon_{2}}\right)\|\sigma(\textbf{0}, k)\|_{\rm H S}^{2}\right] \\
			&\leq  \left(\alpha(k)+\frac{\varepsilon_1}{2}\right)|\varphi(0)-G(\varphi, k)|^{2}+\left(\beta(k)+\gamma_p+\varepsilon_{2} \gamma_p\right) \int_{-\infty}^{0}|\varphi(\theta)|^{2} \rho(\d \theta) \\
			&\quad +\frac{1}{2 \varepsilon_{1}} | b(\textbf{0}, k)|^{2}+\frac{p-1}{2}\left(1+\frac{1}{\varepsilon_{2}}\right)\|\sigma(\textbf{0}, k)\|_{\rm H S}^{2}.
		\end{aligned}
		\]
		For any $\varepsilon>0$, letting \( \varepsilon_{1}=2 \varepsilon, \varepsilon_{2}=\varepsilon/\gamma_p \), we have
		\begin{equation}\label{eq:qianya}
			\begin{aligned}
				&\langle\varphi(0)-G\left(\varphi, k\right), b(\varphi, k)\rangle+\frac{p-1}{2}\left\|\sigma\left(\varphi,k\right)\right\|_{\rm H S}^{2} \\
				&\quad\leq C_{1}(\varepsilon)+(\alpha(k)+\varepsilon)|\varphi(0)-G(\varphi, k)|^{2}+(\beta(k)+\gamma_p+\varepsilon) \int_{-\infty}^{0}|\varphi(\theta)|^{2} \rho(\d \theta),
			\end{aligned}
		\end{equation}
		where
		\[
		C_{1}(\varepsilon)=\frac{1}{4\varepsilon} \max _{1\leq  k\leq  N}|b(\textbf{0}, k)|^{2}+\frac{p-1}{2}\left(1+\frac{\gamma_p}{\varepsilon}\right)\max _{1\leq  k\leq  N}\|\sigma(\textbf{0}, k)\|_{\rm H S}^{2}.
		\]
		
		Next, for every integer $n\geq 1$, define the stopping time 
		\begin{equation}\label{Eq:12201}
			\tau^{\prime}_n:=\inf\{t\geq 0:|Y^{\xi,i}(t)|\geq n\}.
		\end{equation}
	Moreover, there exists a transtion function $P(\omega_{2},\d \omega_{1})$ on $\Omega_2\times\mathscr{B}\left(\Omega_{1}\right)$ such that
		$ \mathbb{P}(\d\omega_{1},\d\omega_{2})= \mathbb{P}_2(\d\omega_{2})P(\omega_{2},\d \omega_{1})$. In the sequel, fix $\omega_{2}\in\Omega_{2}$ and take $\lambda\in (0,pr)$.  Set $$I(t):= \mathbb{E}_{P}\left(\e^{\int_{0}^{t}(\lambda-p\alpha(\Lambda^i(u)))\d u}|Y^{\xi,i}(t)|^{p}  \right),$$ where $\mathbb{E}_{P}$ denotes the expectation with respect to the transition function $P(\omega_{2},\d \omega_{1})$. 
		Moreover, applying the  Itô formula to
		$
		\e^{\int_{0}^{t}(\lambda-p\alpha(\Lambda^i(u)))\mrm du}|Y^{\xi,i}(t)|^p
		$
		and taking expectation with respect to $P$, we derive from \eqref{eq:qianya} that for any $\varepsilon>0$,
		\begin{equation*}
			\begin{aligned}
				I(t\wedge \tau^{\prime}_n)
				&= |Y^{\xi,i}(0)|^{p}+\mathbb{E}_{P} \int_{0}^{t\wedge \tau^{\prime}_n} \e^{\int_{0}^{u}(\lambda-p\alpha(\Lambda^i(v)))\d v }\bigg(\big(\lambda-p\alpha(\Lambda^i(u))\big)|Y^{\xi,i}(u)|^{p} \\
				&\quad+p|Y^{\xi,i}(u)|^{p-2}\langle Y^{\xi,i}(u),b(X^{\xi,i}_u,\Lambda^{i}(u))\rangle+\frac{p}{2}|Y^{\xi,i}(u)|^{p-2}\|\sigma(X^{\xi,i}_u,\Lambda^{i}(u))\|_{\text{HS}}^2\\
				&\quad+\frac{p(p-2)}{2}|Y^{\xi,i}(u)|^{p-4}|(Y^{\xi,i}(u))^{\top}\sigma(X^{\xi,i}_u,\Lambda^{i}(u))|^2 \bigg) \mrm d u\\
				&\leq  |Y^{\xi,i}(0)|^{p}+\mathbb{E}_{P} \int_{0}^{t\wedge \tau^{\prime}_n} \e^{\int_{0}^{u}(\lambda-p\alpha(\Lambda^i(v)))\d v }\bigg(\big(\lambda-p\alpha(\Lambda^i(u))\big)|Y^{\xi,i}(u)|^{p} \\
				&\quad+p|Y^{\xi,i}(u)|^{p-2}\Big(\langle Y^{\xi,i}(u),b(X^{\xi,i}_u,\Lambda^{i}(u))\rangle
				+\frac{p-1}{2}\|\sigma(X^{\xi,i}_u,\Lambda^{i}(u))\|_{\text{HS}}^2\Big)\bigg)\d u\\
				&\leq  |Y^{\xi,i}(0)|^{p}+(\lambda+p\varepsilon)\mathbb{E}_{P} \int_{0}^{t\wedge \tau^{\prime}_n} \e^{\int_{0}^{u}(\lambda-p\alpha(\Lambda^i(v)))\d v }|Y^{\xi,i}(u)|^{p}\d u \\
				&\quad+pC_1{\varepsilon}\mathbb{E}_{P} \int_{0}^{t\wedge \tau^{\prime}_n} \e^{\int_{0}^{u}(\lambda-p\alpha(\Lambda^i(v)))\d v }|Y^{\xi,i}(u)|^{p-2}\d u\\
				&\quad+p(\bar{\beta}+\gamma_p+\varepsilon)\mathbb{E}_{P} \int_{0}^{t\wedge \tau^{\prime}_n}\int_{-\infty}^0 \e^{\int_{0}^{u}(\lambda-p\alpha(\Lambda^i(v)))\d v }|Y^{\xi,i}(u)|^{p-2}|X^{\xi,i}(u+\theta)|^{2}\rho(\d\theta)\d u\\
				&=:\sum_{i=1}^4 I_i
			\end{aligned}
		\end{equation*}
		Now let us estimate these main items above one by one. For $I_1$, by \eqref{eq:xiangai} and \eqref{eq:jianbang}, we are able to obtain
		\begin{equation*}
			\begin{aligned}
				I_1&\leq \left(1+\varepsilon^{\frac{1}{p-1}}\right)^{p-1}\left(|\xi(0)|^p+\frac{1}{\varepsilon}G(\xi,i)|^p\right)
				\leq \left(1+\varepsilon^{\frac{1}{p-1}}\right)^{p-1}\left(|\xi(0)|^p+\frac{\kappa_1^p}{\varepsilon}\|\xi\|_r^p\right).
			\end{aligned}
		\end{equation*}
		Letting $\varepsilon=\kappa_1^{p-1}$ implies
		\begin{equation*}
			I_1\leq \big(1+\kappa_1\big)^p \|\xi\|_r^p\leq 2^p\|\xi\|_r^p.
		\end{equation*}
		For $I_3$,   the Young inequality  yields for any $\varepsilon>0$ and $a\in \mathbb{R}$,
		\[
		|a|^{p-2}=\varepsilon^{\frac{2-p}{p}}(\varepsilon |a|^p)^{\frac{p-2}{p}}\leq \frac{2}{p}\varepsilon^{\frac{2-p}{2}}+\frac{p-2}{p}\varepsilon|a|^p,
		\]
		which implies
		\begin{equation*}
			\begin{aligned}
				I_3&\leq C_{p,\varepsilon}\mathbb{E}_{P} \int_{0}^{t\wedge \tau^{\prime}_n} \e^{\int_{0}^{u}(\lambda-p\alpha(\Lambda^i(v)))\d v }\d u +(p-2)C_1(\varepsilon)\varepsilon\mathbb{E}_{P} \int_{0}^{t\wedge \tau^{\prime}_n} \e^{\int_{0}^{u}(\lambda-p\alpha(\Lambda^i(v)))\d v }|Y^{\xi,i}(u)|^{p}\d u
			\end{aligned}
		\end{equation*} 
		For $I_4$, when $p>2$, by the Young inequality, we have
		\begin{equation}\label{eq:xiuxi1}
			\begin{aligned}
				I_4&\leq (p-2)(\bar{\beta}+\gamma_p+\varepsilon)\mathbb{E}_{P} \int_{0}^{t\wedge \tau^{\prime}_n} \e^{\int_{0}^{u}(\lambda-p\alpha(\Lambda^i(v)))\d v }|Y^{\xi,i}(u)|^{p}\d u \\
				&\quad+2(\bar{\beta}+\gamma_p+\varepsilon)\mathbb{E}_{P} \int_{0}^{t\wedge \tau^{\prime}_n}\int_{-\infty}^0 \e^{\int_{0}^{u}(\lambda-p\alpha(\Lambda^i(v)))\d v }|X^{\xi,i}(u+\theta)|^{p}\rho(\d\theta)\d u .
			\end{aligned}
		\end{equation}
		Hence, \eqref{eq:xiuxi1} holds  for any $p\geq 2$. Note that $\lambda\in (0,pr)$. By the Fubini theorem, \eqref{eq:yao} and (\hyperlink{A0}{A0}), we deduce that for any $\varepsilon\in(0,1)$, 
		\begin{equation}\label{eq:haibuzouma1}
			\begin{aligned}
				\int_{0}^{t\wedge \tau^{\prime}_n}& \int_{-\infty}^0\e^{\int_{0}^{u}(\lambda-p\alpha(\Lambda^i(v)))\d v }|X^{\xi,i}(u+\theta)|^{p} \rho(\d\theta)\d u\\
				&=\int_{0}^{t\wedge \tau^{\prime}_n}\int_{-\infty}^{-u} \e^{-p\int_{0}^{u}\alpha(\Lambda^i(v))\d v} \e^{(\lambda-pr)u}\e^{-p r\theta} \e^{p r(u+\theta)}|X^{\xi,i}(u+\theta)|^{p} \rho(\mathrm{d} \theta) \mathrm{d} u\\
				&\quad+\int_{-{t\wedge \tau^{\prime}_n}}^{0}\int_{0}^{{t\wedge \tau^{\prime}_n}+\theta}\e^{\int_{u}^{u-\theta}(\lambda-p\alpha(\Lambda^i(v)))\d v} \e^{\int_{0}^{u}(\lambda-p\alpha(\Lambda^i(v)))\d v}|X^{\xi,i}(u)|^{p} \mathrm{d} u \rho(\mathrm{d} \theta) \\
				&\leq \delta_{pr}(\rho)\|\xi\|_r^p\int_{0}^{t\wedge \tau^{\prime}_n} \e^{-p\int_{0}^{u}\alpha(\Lambda^i(v))\d v} \d u
				+\delta_{pr-p\underline{\alpha}}(\rho)\int_{0}^{t\wedge \tau^{\prime}_n} \e^{\int_{0}^{u}(\lambda-p\alpha(\Lambda^i(v)))\d v }|X^{\xi,i}(u)|^{p} \d u\\
				&\leq  \delta_{pr}(\rho)\|\xi\|_r^p\int_{0}^{t\wedge \tau^{\prime}_n} \e^{-p\int_{0}^{u}\alpha(\Lambda^i(v))\d v} \d u+\frac{\delta_{pr-p\underline{\alpha}}(\rho)}{(1-\varepsilon)^{p-1}}\int_{0}^{t\wedge \tau^{\prime}_n} \e^{\int_{0}^{u}(\lambda-p\alpha(\Lambda^i(v)))\d v }|Y^{\xi,i}(u)|^{p} \d u\\
				&\quad+\frac{\delta_{pr-p\underline{\alpha}}(\rho)}{\varepsilon^{p-1}}\int_{0}^{t\wedge \tau^{\prime}_n} \e^{\int_{0}^{u}(\lambda-p\alpha(\Lambda^i(v)))\d v }|G(X^{\xi,i}_u,\Lambda^i(u))|^{p} \d u\\
				&\leq \delta_{pr}(\rho)\|\xi\|_r^p\int_{0}^{t\wedge \tau^{\prime}_n} \e^{-p\int_{0}^{u}\alpha(\Lambda^i(v))\d v} \d u+ \frac{\delta_{pr-p\underline{\alpha}}(\rho)}{(1-\varepsilon)^{p-1}}\int_{0}^{t\wedge \tau^{\prime}_n} \e^{\int_{0}^{u}(\lambda-p\alpha(\Lambda^i(v)))\d v }|Y^{\xi,i}(u)|^{p} \d u\\
				&\quad+\frac{\kappa^p\delta_{pr-p\underline{\alpha}}(\rho)}{\varepsilon^{p-1}}\int_{0}^{t\wedge \tau^{\prime}_n} \int_{-\infty}^0\e^{\int_{0}^{u}(\lambda-p\alpha(\Lambda^i(v)))\d v. }|X^{\xi,i}(u+\theta)|^{p} \rho(\d\theta)\d u.
			\end{aligned}
		\end{equation}
		Choosing $\varepsilon=\kappa_2\in (0,1)$ implies 
		\begin{equation}\label{eq:xingba1}
			\begin{aligned}
				\int_{0}^{t\wedge \tau^{\prime}_n} &\int_{-\infty}^0\e^{\int_{0}^{u}(\lambda-p\alpha(\Lambda^i(v)))\d v }|X^{\xi,i}(u+\theta)|^{p} \rho(\d\theta)\d u\\
				&\leq \frac{\delta_{pr}(\rho)}{1-\kappa_2} \|\xi\|_r^p\int_0^{t\wedge \tau^{\prime}_n}\e^{-p\int_{0}^{u}\alpha(\Lambda^i(v))\d v} \d u
				+\frac{\delta_{pr-p\underline{\alpha}}(\rho)}{(1-\kappa_2)^{p}}\int_{0}^{t\wedge \tau^{\prime}_n} \e^{\int_{0}^{u}(\lambda-p\alpha(\Lambda^i(v)))\d v }|Y^{\xi,i}(u)|^{p} \d u.
			\end{aligned}
		\end{equation}
		Inserting  \eqref{eq:xingba1} into \eqref{eq:xiuxi1}, we have
		\begin{equation}\label{eq:mubiao}
			\begin{aligned}
				I_4 &\leq 2(\bar{\beta}+\gamma_p+\varepsilon)\frac{\delta_{pr}(\rho)}{1-\kappa_2} \|\xi\|_r^p\mathbb{E}_{P} \int_0^{t\wedge \tau^{\prime}_n}\e^{-p\int_{0}^{u}\alpha(\Lambda^i(v))\d v} \d u\\
				&\quad	+(\bar{\beta}+\gamma_p+\varepsilon)\left(p-2+\frac{2\delta_{pr-p\underline{\alpha}}(\rho)}{(1-\kappa_2)^{p}}\right)\mathbb{E}_{P} \int_{0}^{t\wedge \tau^{\prime}_n} \e^{\int_{0}^{u}(\lambda-p\alpha(\Lambda^i(v)))\d v }|Y^{\xi,i}(u)|^{p} \d u.
			\end{aligned}
		\end{equation}
		Combining the above calculations leads to
		\begin{equation*}
			\begin{aligned}
				I(t\wedge \tau^{\prime}_n)&\leq C_{p,\varepsilon}\|\xi\|_r^p\left(1+  \int_0^{t}\e^{-p\int_{0}^{u}\alpha(\Lambda^i(v))\d v} \d u\right)+C_{p,\varepsilon} \int_{0}^{t} \e^{\int_{0}^{u}(\lambda-p\alpha(\Lambda^i(v)))\d v }\d u\\
				&\quad +C(p,\varepsilon,\lambda) \int_{0}^{t} \mathbb{E}_{P}\left(\e^{\int_{0}^{u\wedge \tau^{\prime}_n}(\lambda-p\alpha(\Lambda^i(v)))\d v }|Y^{\xi,i}(u\wedge \tau^{\prime}_n)|^{p} \right)\d u,
			\end{aligned}
		\end{equation*}
		where
		\[
		\begin{aligned}
			C(p,\varepsilon,\lambda)&=\lambda+(\bar{\beta}+\gamma_p)\Big(p-2+\frac{2\delta_{pr-p\underline{\alpha}}(\rho)}{(1-\kappa_2)^{p}}\Big)+C_2(\varepsilon),\\
			C_2(\varepsilon)&=(p-2)(C_1(\varepsilon)+1)\varepsilon+\left(p+\frac{2\delta_{pr-p\underline{\alpha}}(\rho)}{(1-\kappa_2)^{p}}\right)\varepsilon.
		\end{aligned}
		\]
		Then, by means of the Gronwall inequality (see, e.g., \cite[Lemma A.31]{Yin2010} ), we obtain
		\begin{equation}\label{Eq:1220}
			\begin{aligned}
				I(t\wedge \tau^{\prime}_n)&\leq  C_{p,\varepsilon}\|\xi\|_r^p\e^{C(p,\varepsilon,\lambda)t}\\
				&\quad+C_{p,\varepsilon}\int_0^t\e^{C(p,\varepsilon,\lambda)(t-u)} \left(\|\xi\|_r^p\e^{-p\int_{0}^{u}\alpha(\Lambda^i(v))\d v}+\e^{\int_{0}^{u}(\lambda-p\alpha(\Lambda^i(v)))\d v }\right)\d u.
			\end{aligned}
		\end{equation}
		By  the Markov inequality,  
		\begin{equation*}
			\begin{aligned}
				P\{\omega_{2},\tau_n^{\prime}<t\}&\leq P\left\{\omega_{2},|Y(t\wedge\tau_n^{\prime})|\geq n\right\}\\
				&= P\left\{\omega_{2},\e^{\int_{0}^{t}(\lambda-p\alpha(\Lambda^i(u)))\mrm du}|Y(t\wedge\tau_n^{\prime})|^p\geq \e^{\int_{0}^{t}(\lambda-p\alpha(\Lambda^i(u)))\mrm du}n^p\right\}\\
				&\leq \frac{I(t\wedge\tau_n^{\prime})}{\e^{\int_{0}^{t}(\lambda-p\alpha(\Lambda^i(u)))\mrm du}n^p}.
			\end{aligned}
		\end{equation*}
		It follows from \eqref{Eq:1220}  that   $P\{\omega_{2},\tau_n^{\prime}<t\}\to 0$ as $n\to \infty$.	 Letting $n\to \infty$ in \eqref{Eq:1220} and using Fatou's lemma, we have
		\begin{equation*}
			\begin{aligned}
				\mathbb{E}_{P}|Y^{\xi,i}(t)|^{p}&\leq C_{p,\varepsilon}\|\xi\|_r^p\left(\e^{C_2(\varepsilon)t}\e^{\int_0^tf(\Lambda^i(u))\d u}+\int_0^t \e^{C_2(\varepsilon)(t-u)-\lambda u}\e^{\int_u^tf(\Lambda^i(u))\d v}\d u\right)\\
				& \quad+C_{p,\varepsilon}\int_0^t \e^{C_2(\varepsilon)(t-u)}\e^{\int_u^tf(\Lambda^i(v))\d v}\d u,
			\end{aligned}
		\end{equation*}
		where $f(\cdot)$ is defined in \eqref{eq:07gggh}. Taking expectation with respect to $\P_2$, we have
		\begin{equation*}
			\begin{aligned}
				\mbb E|Y^{\xi,i}(t)|^{p}&\leq C_{p,\varepsilon}\|\xi\|_r^p\left(\e^{C_2(\varepsilon)t}\E_{\mathbb{P}_2}\e^{\int_0^tf(\Lambda^i(u))\d u}+\int_0^t \e^{C_2(\varepsilon)(t-u)-\lambda u}\E_{\mathbb{P}_2}\e^{\int_u^tf(\Lambda^i(u))\d v}\d u\right)\\
				& \quad+C_{p,\varepsilon}\int_0^t \e^{C_2(\varepsilon)(t-u)}\E_{\mathbb{P}_2}\e^{\int_u^tf(\Lambda^i(v))\d v}\d u.
			\end{aligned}
		\end{equation*}
		Thus, letting $\varepsilon$ be sufficiently small such that $C_2(\varepsilon)<\zeta$ , we have
		\begin{equation*}
			\begin{aligned}
				\mbb E|Y^{\xi,i}(t)|^{p}&\leq C_{p,\varepsilon}\|\xi\|_r^p\left(\e^{C_2(\varepsilon)t}\e^{-\zeta t}+\int_0^t \e^{C_2(\varepsilon)(t-u)-\lambda u}\e^{-\zeta(t-u)}\d u\right)\\
				& \quad+C_{p,\varepsilon}\int_0^t \e^{C_2(\varepsilon)(t-u)}\e^{-\zeta(t-u)}\d u\\
				&  \leq C_{p,\varepsilon}(1+\|\xi\|_r^p)
			\end{aligned}
		\end{equation*}
		where the first inequality has used Lemma \ref{eq:wuwei}. The proof is complete.
	\end{proof}
	
	\begin{lemma}\label{lem:xiuxi8}
		Suppose that assumptions of Lemma \ref{thm:opao1} hold. Then there exist constants $C$, $\lambda>0$  such that for any   $\xi$, $\eta \in  C_r$,  $i \in \mbb S $, and  $t\geq 0 $,
		\begin{equation}\label{eq:bushi1}
			\mbb{E}\left|Y^{\xi, i}(t)-Y^{\eta, i}(t)\right|^{p} \leq C \|\xi-\eta\|_{r}^{p}\e^{-\lambda t}.
		\end{equation}      
	\end{lemma}
	\begin{proof}
		The proof  is similar to Lemma \ref{thm:opao1}.   Whereas we herein provide an outline of the argument just to make the content self-contained. For every integer $n\geq 1$, define the stopping time
		\begin{equation}\label{Eq:1104:1}
			\hat{\tau}_n=\inf\{t\geq 0:|Y^{\xi,i}(t)|\vee |Y^{\eta,i}(t)|\geq n\}.
		\end{equation}
		We claim that $\hat{\tau}_n$ tends to $\infty$ almost surely, as $n\to\infty$. In fact, for any given $T>0$, we have $|Y^{\xi,i}(\hat{\tau}_n\wedge T)|^p\vee  |Y^{\eta,i}(\hat{\tau}_n\wedge T)|^p\geq n^p$ on the set $\left\{\hat{\tau}_n \leq T\right\}$. Therefore, it follows from the Markov inequality and Lemma \ref{thm:opao1} that
		\begin{equation}\label{EQ:1030:11}
			\begin{aligned}
				\mathbb{P}\left\{\hat{\tau}_n \leq T\right\}&\leq \mathbb{P}\{|Y^{\xi,i}(\hat{\tau}_n\wedge T)|^p\vee  |Y^{\eta,i}(\hat{\tau}_n\wedge T)|^p\geq n^p\}\\
				&\leq \frac{1  }{n^p}\left(\mathbb{E}|Y^{\xi,i}(\hat{\tau}_n\wedge T)|^p\vee \mathbb{E}|Y^{\eta,i}(\hat{\tau}_n\wedge T)|^p \right)\\
				& \leq \frac{C}{n^p}\left[(1+\|\xi\|_r^p)\vee(1+\|\eta\|_r^p)\right]\rightarrow 0,
			\end{aligned}
		\end{equation}
		as  $n \rightarrow \infty$. Fix $\omega_{2}\in\Omega_{2}$ and take $\lambda\in (0,pr)$. Set $J(t):=\mathbb{E}_{P}\left( \e^{\int_{0}^{t}(\lambda-p\alpha(\Lambda^i(u)))\d u}|Y^{\xi,i}(t)-Y^{\eta,i}(t)|^{p}\right)$.  
		Applying the  Itô formula to
		$
		\e^{\int_{0}^{t}(\lambda-p\alpha(\Lambda^i(u)))\mrm du}|Y^{\xi,i}(t)-Y^{\eta,i}(t)|^p
		$
		and taking expectation with respect to the transition function $P(\omega_{2},\d \omega_{1})$, we derive from (\hyperlink{A1}{A1}) and (\hyperlink{A2}{A2}) that
		\begin{equation*}
			\begin{aligned}
				J&(t\wedge\hat{\tau}_n)\\
				&\leq  |Y^{\xi,i}(0)-Y^{\eta,i}(0)|^{p}+\mathbb{E}_{P} \int_{0}^{t\wedge\hat{\tau}_n} \e^{\int_{0}^{u}(\lambda-p\alpha(\Lambda^i(v)))\d v }\bigg(\big(\lambda-p\alpha(\Lambda^i(u))\big)|Y^{\xi,i}(u)-Y^{\eta,i}(u)|^{p} \\
				&\quad+p|Y^{\xi,i}(u)-Y^{\eta,i}(u)|^{p-2}\Big(\langle Y^{\xi,i}(u)-Y^{\eta,i}(u),b(X^{\xi,i}_u,\Lambda^{i}(u))-b(X^{\eta,i}_u,\Lambda^{i}(u))\rangle\\
				&\quad+\frac{p-1}{2}\|\sigma(X^{\xi,i}_u,\Lambda^{i}(u))-\sigma(X^{\eta,i}_u,\Lambda^{i}(u))\|_{\text{HS}}^2\Big)\bigg)\d u\\
				&\leq  |Y^{\xi,i}(0)-Y^{\eta,i}(0)|^{p}+\lambda\mathbb{E}_{P} \int_{0}^{t\wedge\hat{\tau}_n} \e^{\int_{0}^{u}(\lambda-p\alpha(\Lambda^i(v)))\d v }|Y^{\xi,i}(u)-Y^{\eta,i}(u)|^{p}\d u+p(\bar{\beta}+\gamma_p) \\
				&\quad\times\mathbb{E}_{P} \int_{0}^{t\wedge\hat{\tau}_n}\int_{-\infty}^0 \e^{\int_{0}^{u}(\lambda-p\alpha(\Lambda^i(v)))\d v }|Y^{\xi,i}(u)-Y^{\eta,i}(u)|^{p-2}|X^{\xi,i}(u+\theta)-X^{\eta,i}(u+\theta)|^{2}\rho(\d\theta)\d u\\
				&=:\sum_{i=1}^3 J_i
			\end{aligned}
		\end{equation*}
		Similar to  \eqref{eq:mubiao},  we have
		\begin{equation*}
			\begin{aligned}
				J_3 &\leq 2(\bar{\beta}+\gamma_p)\frac{\delta_{pr}(\rho)}{1-\kappa_2} \|\xi-\eta\|_r^p\mathbb{E}_{P} \int_0^{t\wedge\hat{\tau}_n}\e^{-p\int_{0}^{u}\alpha(\Lambda^i(v))\d v} \d u\\
				&\quad	+(\bar{\beta}+\gamma_p)\left(p-2+\frac{2\delta_{pr-p\underline{\alpha}}(\rho)}{(1-\kappa_2)^{p}}\right)\mathbb{E}_{P} \int_{0}^{t\wedge\hat{\tau}_n} \e^{\int_{0}^{u}(\lambda-p\alpha(\Lambda^i(v)))\d v }|Y^{\xi,i}(u)-Y^{\eta,i}(u)|^{p} \d u,
			\end{aligned}
		\end{equation*}
		which implies 
		\begin{equation*}
			\begin{aligned}
				J(t\wedge\hat{\tau}_n)&\leq C_{p}\|\xi-\eta\|_r^p\left(1+ \int_0^{t}\e^{-p\int_{0}^{u}\alpha(\Lambda^i(v))\d v} \d u\right)\\
				&\quad +C(p,\lambda) \int_{0}^{t} \e^{\int_{0}^{u\wedge\hat{\tau}_n}(\lambda-p\alpha(\Lambda^i(v)))\d v }\mathbb{E}_{P}|Y^{\xi,i}(u\wedge\hat{\tau}_n)-Y^{\eta,i}(u\wedge\hat{\tau}_n)|^{p} \d u,
			\end{aligned}
		\end{equation*}
		where
		\[
		C(p,\lambda)=\lambda+(\bar{\beta}+\gamma_p)\Big(p-2+\frac{2\delta_{pr-p\underline{\alpha}}(\rho)}{(1-\kappa_2)^{p}}\Big).
		\]
		Hence, the Gronwall inequality implies 
		\begin{equation*}
			\begin{aligned}
				J(t\wedge\hat{\tau}_n)&\leq  C_{p}\|\xi-\eta\|_r^p\left(\e^{C(p,\lambda)t}+\int_0^t\e^{C(p,\lambda)(t-u)} \e^{-p\int_{0}^{u}\alpha(\Lambda^i(v))\d v}\d u\right).
			\end{aligned}
		\end{equation*}
		Letting $n\to\infty$ and utilizing Fatou's lemma yields 
		\begin{equation*}
			\begin{aligned}
				\e&^{\int_{0}^{t}(\lambda-p\alpha(\Lambda^i(u)))\d u} \mathbb{E}_{P}|Y^{\xi,i}(t)-Y^{\eta,i}(t)|^{p}\\
				&\leq  C_{p}\|\xi-\eta\|_r^p\left(\e^{C(p,\lambda)t}+\int_0^t\e^{C(p,\lambda)(t-u)} \e^{-p\int_{0}^{u}\alpha(\Lambda^i(v))\d v}\d u\right),
			\end{aligned}
		\end{equation*}
		which implies 
		\begin{equation*}
			\begin{aligned}
				\mathbb{E}_{P}|Y^{\xi,i}(t)-Y^{\eta,i}(t)|^{p}&\leq C_{p}\|\xi-\eta\|_r^p\left(\e^{\int_0^tf(\Lambda^i(u))\d u}+\int_0^t \e^{-\lambda u}\e^{\int_u^tf(\Lambda^i(u))\d v}\d u\right).
			\end{aligned}
		\end{equation*}
		Taking expectation with respect to $\P_2$, we have
		\begin{equation*}
			\begin{aligned}
				\mbb E|Y^{\xi,i}(t)-Y^{\eta,i}(t)|^{p}&\leq C_{p}\|\xi-\eta\|_r^p\left(\E_{\P_2}\e^{\int_0^tf(\Lambda^i(u))\d u}+\int_0^t \e^{-\lambda u}\E_{\P_2}\e^{\int_u^tf(\Lambda^i(u))\d v}\d u\right)
			\end{aligned}
		\end{equation*}
		Thus,  choosing  $\lambda\in (0,pr\wedge\zeta)$ , we have
		\begin{equation*}
			\begin{aligned}
				\mbb E|Y^{\xi,i}(t)-Y^{\eta,i}(t)|^{p}&\leq C_{p}\|\xi-\eta\|_r^p\left(\e^{-\zeta t}+\int_0^t \e^{-\lambda u}\e^{-\zeta(t-u)}\d u\right)\leq C_{p}\|\xi-\eta\|_r^p\e^{-\lambda t},
			\end{aligned}
		\end{equation*}
		where the first inequality has used Lemma \ref{eq:wuwei}. The proof is complete.
	\end{proof}
	\subsection{Boundedness and Convergence of $X_t$}
	\begin{theorem}\label{thm:opaoooo}
		Suppose that assumptions of Lemma \ref{thm:opao1} hold.  Then there exists a constant $C>0$  such that for any   $(\xi,i)\in E$ and  $t\geq 0 $,
		\begin{equation}\label{eq:bushi}
			\mbb{E}\|X^{\xi, i}_t\|_r^{p} \leq  {C}(1+\|\xi\|_{r}^{p}).
		\end{equation}
	\end{theorem}
	\begin{proof}
		By the definition of $\|\cdot\|_r$ and some basic calculations, we get
		\begin{equation}\label{eq:bumingbai1}
			\e^{p r t}\mbb E\|X_{t}^{\xi,i}\|_{r}^{p} \leq  \|\xi\|_{r}^{p}+\mbb E\left( \sup _{0 < u \leq t}\e^{p r u}\left|X^{\xi,i}(u)\right|^{p}\right).
		\end{equation}
		By \eqref{eq:yao} and (\hyperlink{A0}{A0}), we have
		\begin{equation*}
			\begin{aligned}
				\sup _{0 < u \leq t}\e^{p r u}\left|X^{\xi,i}(u)\right|^{p}&\leq \frac{1}{(1-\varepsilon)^{p-1}}\sup _{0 < u \leq t}\e^{p r u}\left|Y^{\xi,i}(u)\right|^{p}+\frac{\kappa^p}{\varepsilon^{p-1}}\sup _{0 < u \leq t}\int_{-\infty}^{0}\e^{p r u}|X^{\xi,i}(u+\theta)|^{p}\rho(\d \theta)\\
				&\leq  \frac{1}{(1-\varepsilon)^{p-1}}\sup _{0 < u \leq t}\e^{p r u}\left|Y^{\xi,i}(u)\right|^{p}+\frac{\kappa^p\delta_{pr}(\rho)}{\varepsilon^{p-1}}\left(\|\xi\|_{r}^{p}+\sup _{0 < u \leq t}\e^{p r u}\left|X^{\xi,i}(u)\right|^{p}\right).
			\end{aligned}
		\end{equation*}
		Letting $\varepsilon=\kappa_1$ implies
		\begin{equation}\label{eq:putaogan1}
			\begin{aligned}
				\sup _{0 < u \leq t}\e^{p r u}\left|X^{\xi,i}(u)\right|^{p}
				\leq & \frac{\kappa_1}{1-\kappa_1}\|\xi\|_{r}^{p} +  \frac{1}{(1-\kappa_1)^{p}}\sup _{0 < u \leq t}\e^{p r u}\left|Y^{\xi,i}(u)\right|^{p}.
			\end{aligned} 
		\end{equation}
		Inserting \eqref{eq:putaogan1} into \eqref{eq:bumingbai1} yields
		\begin{equation}\label{eq:weiwei}
			\e^{p r t}\mbb E\|X_{t}^{\xi,i}\|_{r}^{p} \\
			\leq \frac{1}{1-\kappa_1}\|\xi\|_{r}^{p} + \frac{1}{(1-\kappa_1)^{p-1}}\E\left( \sup _{0 < u \leq t}\e^{p r u}\left|Y^{\xi,i}(u)\right|^{p}\right).
		\end{equation}
		Set $\Gamma(t):=\mathbb{E}\left(\sup_{0<u \leq t} \mathrm{e}^{p r u}\left|Y^{\xi,i}(u)\right|^{p}\right)$.  The estimation of $\Gamma(t)$ is given as follows. 	Applying Itô formula to $\e^{p r t}|Y^{\xi,i}(t)|^{p} $ and using \eqref{eq:qianya}, we obtain
		\begin{equation*}
			\begin{aligned}
				\e^{p r t}|Y^{\xi,i}(t)|^{p} 
				&\leq |Y^{\xi,i}(0)|^p+\int_{0}^t  \e^{p r u}\bigg( pr|Y^{\xi,i}(u)|^{p}+p|Y^{\xi,i}(u)|^{p-2}\Big(\langle Y^{\xi,i}(u),b(X^{\xi,i}_u,\Lambda^{i}(u))\rangle\\
				&\quad+\frac{p-1}{2}\|\sigma(X^{\xi,i}_u,\Lambda^{i}(u))\|_{\text{HS}}^2\Big)\bigg)\d u+p \int_{0}^{t} \e^{p r u}|Y^{\xi,i}(u)|^{p-2}\langle Y^{\xi,i}(u),\sigma(X^{\xi,i}_u, \Lambda(u)) \mathrm{d} W(u)\rangle\\ 
				&\leq |Y^{\xi,i}(0)|^p+p(r+\bar{\alpha}\vee 0+\varepsilon )\int_{0}^{t} \e^{p r u}|Y^{\xi,i}(u)|^{p} \mrm d u +C_{p,\varepsilon}\int_{0}^{t} \e^{p r u}|Y^{\xi,i}(u)|^{p-2} \mrm d u 
				\\
				&\quad+p(\bar{\beta}+\gamma_p+\varepsilon) \int_{0}^{t} \int_{-\infty}^{0} \e^{p r u}|Y^{\xi,i}(u)|^{p-2}|X^{\xi,i}(u+\theta)|^{2} \rho(\mathrm{d} \theta) \mathrm{d} u\\
				&\quad+p \int_{0}^{t} \e^{p r u}|Y^{\xi,i}(u)|^{p-2}\left\langle Y^{\xi,i}(u),\sigma(X^{\xi,i}_u, \Lambda^{i}(u)) \mathrm{d} W(u)\right\rangle.
			\end{aligned}
		\end{equation*}
		Moreover, we derive from \eqref{eq:qianya} that 
		\[
		\frac{p-1}{2}\|\sigma\left(\varphi,k\right)\|_{\rm H S}^{2}\leq C_{1}(\varepsilon)+(\gamma_p+\varepsilon) \int_{-\infty}^{0}|\varphi(\theta)|^{2} \rho(\d \theta),
		\]
		It follows from the BDG inequality that
		\begin{equation*}
			\begin{aligned}
				\mathbb{E}&\left(\sup_{0<u \leq t\wedge\tau^{\prime}_n} \int_{0}^{u} \e^{p r v}|Y^{\xi,i}(v)|^{p-2}\left\langle Y^{\xi,i}(v),\sigma(X^{\xi,i}_v, \Lambda^{i}(v)) \mathrm{d} W(v)\right\rangle\right)\\
				&\leq 4\sqrt{2}\mathbb{E}\left(\int_0^{t\wedge\tau^{\prime}_n}\e^{2pru}|Y^{\xi,i}(u)|^{2p-2}\|\sigma(X^{\xi,i}_u, \Lambda^{i}(u)) \|_{\rm HS}^2\d u\right)^{\frac{1}{2}}\\
				&\leq \frac{1}{2}\Gamma(t\wedge\tau^{\prime}_n)+C\E\int_0^{t\wedge\tau^{\prime}_n}\e^{pru}|Y^{\xi,i}(u)|^{p-2}\|\sigma(X^{\xi,i}_u, \Lambda^{i}(u)) \|_{\rm HS}^2\d u\\
				&\leq \frac{1}{2}\Gamma(t\wedge\tau^{\prime}_n)+C_{p,\varepsilon}\E\int_{0}^{t\wedge\tau^{\prime}_n} \e^{p r u}|Y^{\xi,i}(u)|^{p-2} \mrm d u \\
				&\quad+C_{p,\varepsilon}\E\int_{0}^{t\wedge\tau^{\prime}_n} \int_{-\infty}^{0} \e^{p r u}|Y^{\xi,i}(u)|^{p-2}|X^{\xi,i}(u+\theta)|^{2} \rho(\mathrm{d} \theta) \mathrm{d} u,\\
			\end{aligned}
		\end{equation*}
		where $\tau^{\prime}_n$  is given in \eqref{Eq:12201}. The Young inequality implies 
			\begin{align*}
				\Gamma&(t\wedge\tau^{\prime}_n)\\
				&\leq |Y^{\xi,i}(0)|^p+C_{p,\varepsilon}\E\int_{0}^{t\wedge\tau^{\prime}_n} \e^{p r u}|Y^{\xi,i}(u)|^{p} \mrm d u +C_{p,\varepsilon}\E\int_{0}^{t\wedge\tau^{\prime}_n} \e^{p r u}|Y^{\xi,i}(u)|^{p-2} \mrm d u \\
				&\quad+C_{p,\varepsilon}\E\int_{0}^{t\wedge\tau^{\prime}_n} \int_{-\infty}^{0} \e^{p r u}|Y^{\xi,i}(u)|^{p-2}|X^{\xi,i}(u+\theta)|^{2} \rho(\mathrm{d} \theta) \mathrm{d} u\\
				&\quad+p \mathbb{E}\left(\sup_{0<u \leq t\wedge\tau^{\prime}_n} \int_{0}^{u} \e^{p r v}|Y^{\xi,i}(v)|^{p-2}\left\langle Y^{\xi,i}(v),\sigma(X^{\xi,i}_v, \Lambda^{i}(v)) \mathrm{d} W(v)\right\rangle\right)\\
				& \leq |Y^{\xi,i}(0)|^p+\frac{1}{2}\Gamma(t\wedge\tau^{\prime}_n)+C_{p,\varepsilon}\E\int_{0}^{t\wedge\tau^{\prime}_n} \e^{p r u}|Y^{\xi,i}(u)|^{p} \mrm d u +C_{p,\varepsilon}\E\int_{0}^{t\wedge\tau^{\prime}_n} \e^{p r u}|Y^{\xi,i}(u)|^{p-2} \mrm d u \\
				&\quad+C_{p,\varepsilon}\E\int_{0}^{t\wedge\tau^{\prime}_n} \int_{-\infty}^{0} \e^{p r u}|Y^{\xi,i}(u)|^{p-2}|X^{\xi,i}(u+\theta)|^{2} \rho(\mathrm{d} \theta) \mathrm{d} u\\
				&\leq |Y^{\xi,i}(0)|^p+\frac{1}{2}\Gamma(t\wedge\tau^{\prime}_n)+C_{p,\varepsilon}\e^{prt}+C_{p,\varepsilon}\E\int_{0}^{t\wedge\tau^{\prime}_n} \e^{p r u}|Y^{\xi,i}(u)|^{p} \mrm d u\\
				&\quad+C_{p,\varepsilon}\int_{0}^{t\wedge\tau^{\prime}_n} \int_{-\infty}^{0} \e^{p r u}|X^{\xi,i}(u+\theta)|^{p} \rho(\mathrm{d} \theta) \mathrm{d} u,
			\end{align*}
		Similar to \eqref{eq:haibuzouma1}, we have 
		\begin{equation*}
			\begin{aligned}
				\int_{0}^{t\wedge\tau^{\prime}_n} &\int_{-\infty}^{0} \e^{pr u}|X^{\xi,i}(u+\theta)|^{p} \rho(\mathrm{d} \theta) \mathrm{d} u \\
				&=\int_{0}^{t\wedge\tau^{\prime}_n} \int_{-\infty}^{-u} \e^{-p r\theta} \e^{p r(u+\theta)}|X^{\xi,i}(u+\theta)|^{p} \rho(\mathrm{d} \theta) \mathrm{d} u+\int_{-t\wedge\tau^{\prime}_n}^{0}\int_{0}^{t\wedge\tau^{\prime}_n+\theta} \e^{pr(u-\theta)}|X^{\xi,i}(u)|^{p} \mathrm{d} u \rho(\mathrm{d} \theta) \\
				&\leq \delta_{pr}\left(\rho\right)\|\xi\|_{r}^{p}t+\delta_{pr}\left(\rho\right) \int_{0}^{t\wedge\tau^{\prime}_n}\e^{pr u}|X^{\xi,i}(u)|^p\mathrm{d}u\\
				&\leq \delta_{pr}\left(\rho\right)\|\xi\|_{r}^{p}t+\frac{\delta_{pr}\left(\rho\right)}{(1-\varepsilon)^{p-1}} \int_{0}^{t\wedge\tau^{\prime}_n}\e^{pr u}|Y^{\xi,i}(u)|^p\mathrm{d}u\\
				&\quad+\frac{\delta_{pr}\kappa^p\left(\rho\right)}{\varepsilon^{p-1}} \int_{0}^{t\wedge\tau^{\prime}_n} \int_{-\infty}^{0} \e^{pr u}|X^{\xi,i}(u+\theta)|^{p} \rho(\mathrm{d} \theta) \mathrm{d} u 
			\end{aligned}
		\end{equation*}
		for any $\varepsilon>0$.	Letting $\varepsilon=\kappa_1$ implies that 
		\begin{equation}\label{eq:xixifu}
			\int_{0}^{t\wedge\tau^{\prime}_n} \int_{-\infty}^{0} \e^{pr u}|X^{\xi,i}(u+\theta)|^{p} \rho(\mathrm{d} \theta) \mathrm{d} u 
			\leq \frac{\delta_{pr}\left(\rho\right)}{1-\kappa_1}\|\xi\|_{r}^{p}t+\frac{\delta_{pr}\left(\rho\right)}{(1-\kappa_1)^{p}} \int_{0}^{t\wedge\tau^{\prime}_n}\e^{pr u}|Y^{\xi,i}(u)|^p\mathrm{d}u
		\end{equation}
		Combining the above calculations, we get
		\begin{equation*}
			\begin{aligned}
				\Gamma(t\wedge\tau^{\prime}_n)&\leq C_p\|\xi\|_r^p(1+t)+C_{p,\varepsilon}\e^{prt}+C_{p,\varepsilon}\E\int_{0}^{t\wedge\tau^{\prime}_n} \e^{p r u}|Y^{\xi,i}(u)|^{p} \mrm d u.
			\end{aligned}	
		\end{equation*}
		By a similar argument as  \eqref{EQ:1030:11},  $\tau^{\prime}_n$ tends to $\infty$, as $n\to\infty$.   Letting $n\to\infty$ and using Fatou's lemma implies 
		\begin{equation*}
			\begin{aligned}
				\Gamma(t)&\leq C_p\|\xi\|_r^p(1+t)+C_{p,\varepsilon}\e^{prt}+C_{p,\varepsilon}\E\int_{0}^{t} \e^{p r u}|Y^{\xi,i}(u)|^{p} \mrm d u\\
				&\leq C_{p,\varepsilon}\e^{prt}+C_p\|\xi\|_r^p(1+t+\e^{prt}),
			\end{aligned}	
		\end{equation*}
		where Lemma \ref{thm:opao1} has been used in the last inequality. Thus, 
		\begin{equation*}
			\begin{aligned}
				\mbb E\|X_{t}^{\xi,i}\|_{r}^{p}& \leq \e^{-p r t} \|\xi\|_{r}^{p}+\e^{-p r t}\Gamma(t)\\
				&\leq C_{p,\varepsilon}+ \e^{-p r t} \|\xi\|_{r}^{p}+C_p\|\xi\|_r^p(1+\e^{-p r t}+t\e^{-p r t})\\
				&\leq C_{p,\varepsilon}(1+\|\xi\|_r^p),
			\end{aligned}
		\end{equation*}
		where we have used the fact that for any $c>0$ and $t>0$, $t\e^{-ct}\leq \frac{1}{c\e}$.  Now the proof is complete.\end{proof}
	
	\begin{theorem}\label{thm:opao2}
		Suppose that assumptions of Lemma \ref{thm:opao1} hold. Then there exists a constant $C>0$  such that for any   $\xi$, $\eta \in  C_r$,  $i \in \mbb S $, and  $t\geq 0 $,
		\begin{equation}\label{eq:hengaoxing}
			\mbb{E}\|X^{\xi, i}_t-X^{\eta, i}_t\|_r^{p} \leq C \|\xi-\eta\|_{r}^{p}\e^{-\lambda t},
		\end{equation}	
		where $\lambda$ is given in \eqref{eq:bushi1}.
	\end{theorem}
	\begin{proof}
		The proof  is similar to Theorem \ref{thm:opaoooo}. First, we have
		\begin{equation}\label{eq:bumingbai}
			\begin{aligned}
				\e^{p r t}\mbb E\|X_{t}^{\xi,i}-X_{t}^{\eta,i}\|_{r}^{p}& \leq  \|\xi-\eta\|_{r}^{p}+\mbb E\left( \sup _{0 < u \leq t}\e^{p r u}\left|X^{\xi,i}(u)-X^{\eta,i}(u)\right|^{p}\right)\\
				& \leq \frac{1}{1-\kappa_1}\|\xi-\eta\|_{r}^{p} + \frac{1}{(1-\kappa_1)^{p-1}}\E\left( \sup _{0 < u \leq t}\e^{p r u}\left|Y^{\xi,i}(u)-Y^{\eta,i}(u)\right|^{p}\right).
			\end{aligned}
		\end{equation}
		Set $\Theta(t):=\mathbb{E}\left(\sup_{0<u \leq t} \mathrm{e}^{p r u}\left|Y^{\xi,i}(u)-Y^{\eta,i}(u)\right|^{p}\right)$. Applying Itô formula to $\e^{p r t}|Y^{\xi,i}(t)-Y^{\eta,i}(t)|^{p} $ and using (\hyperlink{A1}{A1}) and (\hyperlink{A2}{A2}), we obtain
		\begin{equation*}
			\begin{aligned}
				\e^{p r t}&|Y^{\xi,i}(t)-Y^{\eta,i}(t)|^{p}\\ 
				&\leq |Y^{\xi,i}(0)-Y^{\eta,i}(0)|^p+p(r+\bar{\alpha}\vee 0)\int_{0}^{t} \e^{p r u}|Y^{\xi,i}(u)-Y^{\eta,i}(u)|^{p} \mrm d u \\
				&+p(\bar{\beta}+\gamma_p) \int_{0}^{t} \int_{-\infty}^{0} \e^{p r u}|Y^{\xi,i}(u)-Y^{\eta,i}(u)|^{p-2}|X^{\xi,i}(u+\theta)-X^{\eta,i}(u+\theta)|^{2} \rho(\mathrm{d} \theta) \mathrm{d} u\\
				&+p \int_{0}^{t} \e^{p r u}|Y^{\xi,i}(u)-Y^{\eta,i}(u)|^{p-2}\left\langle Y^{\xi,i}(u)-Y^{\eta,i}(u),\left(\sigma(X^{\xi,i}_u, \Lambda^{i}(u))-\sigma(X^{\eta,i}_u, \Lambda^{i}(u))\right) \mathrm{d} W(u)\right\rangle.
			\end{aligned}
		\end{equation*}
		Hence, by the Young inequality and the BDG inequality,  we get 
		\begin{equation}\label{Eq:1104:2}
			\begin{aligned}
				\Theta(t\wedge\hat{\tau}_n) &\leq |Y^{\xi,i}(0)-Y^{\eta,i}(0)|^p+C_{p,\varepsilon}\E\int_{0}^{t\wedge\hat{\tau}_n} \e^{p r u}|Y^{\xi,i}(u)-Y^{\eta,i}(u)|^{p} \mrm d u\\
				&\quad+C_{p}\E\int_{0}^{t\wedge\hat{\tau}_n} \int_{-\infty}^{0} \e^{p r u}|Y^{\xi,i}(u)-Y^{\eta,i}(u)|^{p-2}|X^{\xi,i}(u+\theta)-Y^{\eta,i}(u+\theta)|^{2} \rho(\mathrm{d} \theta) \mathrm{d} u\\
				&\quad+p \mathbb{E}\left(\sup_{0<u \leq t\wedge\hat{\tau}_n} \int_{0}^{u} \e^{p r v}|Y^{\xi,i}(v)-Y^{\eta,i}(v)|^{p-2}\right.\\
				&\qquad\qquad\left.\times\left\langle Y^{\xi,i}(u)-Y^{\eta,i}(u),\left(\sigma(X^{\xi,i}_u, \Lambda^{i}(u))-\sigma(X^{\eta,i}_u, \Lambda^{i}(u))\right) \mathrm{d} W(u)\right\rangle\right)\\
				&\leq |Y^{\xi,i}(0)-Y^{\eta,i}(0)|^p+C_{p}\E\int_{0}^{t\wedge\hat{\tau}_n} \e^{p r u}|Y^{\xi,i}(u)-Y^{\eta,i}(u)|^{p} \mrm d u  \\
				&\quad+C_{p}\E\int_{0}^{t\wedge\hat{\tau}_n} \int_{-\infty}^{0} \e^{p r u}|Y^{\xi,i}(u)-Y^{\eta,i}(u)|^{p-2}|X^{\xi,i}(u+\theta)-X^{\eta,i}(u+\theta)|^{2} \rho(\mathrm{d} \theta) \mathrm{d} u\\
				&\leq |Y^{\xi,i}(0)-Y^{\eta,i}(0)|^p+C_{p}\E\int_{0}^{t\wedge\hat{\tau}_n} \e^{p r u}|Y^{\xi,i}(u)-Y^{\eta,i}(u)|^{p} \mrm d u\\
				&\quad+C_{p}\E\int_{0}^{t\wedge\hat{\tau}_n} \int_{-\infty}^{0} \e^{p r u}|X^{\xi,i}(u+\theta)-X^{\eta,i}(u+\theta)|^{p} \rho(\mathrm{d} \theta) \mathrm{d} u,
			\end{aligned}
		\end{equation}
		where $\hat{\tau}_n$ is defined in \eqref{Eq:1104:1}. Similar to \eqref{eq:xixifu}, we have
		\begin{equation}\label{Eq:1104:3}
			\begin{aligned}
				\int_{0}^{t\wedge\hat{\tau}_n} &\int_{-\infty}^{0} \e^{pr u}|X^{\xi,i}(u+\theta)-X^{\eta,i}(u+\theta)|^{p} \rho(\mathrm{d} \theta) \mathrm{d} u \\
				& \leq \frac{\delta_{pr}\left(\rho\right)}{1-\kappa_1}\|\xi-\eta\|_{r}^{p}t+\frac{\delta_{pr}\left(\rho\right)}{(1-\kappa_1)^{p}} \int_{0}^{t\wedge\hat{\tau}_n}\e^{pr u}|Y^{\xi,i}(u)-Y^{\eta,i}(u)|^p\mathrm{d}u.
			\end{aligned} 
		\end{equation}
		Plugging \eqref{Eq:1104:3} into \eqref{Eq:1104:2} and letting $n\to \infty$ gets
		\begin{equation*}
			\begin{aligned}
				\Theta(t)&\leq C_p\|\xi-\eta\|_r^p(1+t)+C_{p}\E\int_{0}^{t} \e^{p r u}|Y^{\xi,i}(u)-Y^{\eta,i}(u)|^{p} \mrm d u\\
				&\leq C_p\|\xi-\eta\|_r^p(1+t+\e^{(pr-\lambda)t}),
			\end{aligned}	
		\end{equation*}
		where Lemma  \ref{lem:xiuxi8} has been used in the last inequality. Thus, 
		\begin{equation*}
			\begin{aligned}
				\mbb E\|X_{t}^{\xi,i}\|_{r}^{p}& \leq \e^{-p r t} \|\xi-\eta\|_{r}^{p}+\e^{-p r t}\Theta(t)\\
				&\leq C_p\|\xi-\eta\|_r^p(\e^{-p r t}+t\e^{-(p r-\lambda) t}\e^{-\lambda t}+\e^{-\lambda t})\\
				&\leq C_p\|\xi-\eta\|_r^p\e^{-\lambda t},
			\end{aligned}
		\end{equation*}
		where we have used the fact  $\lambda<pr$.  The proof is complete.	
	\end{proof}
	Finally, we conclude this subsection with the following property.
	Let $\delta_{(\xi, i)} P_{t}(\cdot)$, $t\geq0$,  $(\xi,i)\in E$,   be the transition probability function of of $(X_t,\Lambda(t))$. Denote by ${\left(P_{t}\right)_{t \geq 0}}$ the semigroup associated to $\delta_{(\xi, i)} P_{t}(\cdot)$.  To be precise, for $g\in\mathscr{B}_b(E)$, the Markov semigroup $P_t$ is defined by 
	\[
	P_tg(\xi,i)=\E g(X^{\xi, i}(t),\Lambda^{i}(t))=\int_{E}g(\eta,j)\delta_{(\xi, i)} P_{t}(\mathrm{d} \eta, \mrm d j).
	\] 
	\begin{corollary}\label{cor:tuishitou}
		Suppose that assumptions of Lemma \ref{thm:opao1}  hold. Then $(X_t,\Lambda(t))$ admits the Feller property. Moreover, $(X_t,\Lambda(t))$ is a strong Markov process.
	\end{corollary}
	\begin{proof}
		For any $g\in \mathscr{C}_b(E)$, it is obvious that $P_tg(\xi,i)$, $(\xi,i)\in E$, is bounded. Since $\S$ has a discrete metric, it suffices to show that $P_tg(\xi,i)$ is continuous with respect to $\xi$.  By a similar argument as in \cite[Lemma 8.1.4]{oksendal2013stochastic} ,   we only  need to prove that for any  $\xi$, $\eta \in  C_r$,  $i \in \mbb S $, and  $t\geq 0 $,
		\begin{equation*}
			\mbb{E}\|X^{\xi, i}_t-X^{\eta, i}_t\|_r^{p} \leq C(t) \|\xi-\eta\|_{r}^{p},
		\end{equation*}	
		where $C(t)$ does not depend on $\xi$ and $\eta$. Hence, the Feller property is derived from Theorem \ref{thm:opao2}. Using the standard technique (see, e.g.,\cite[Proposition 4.1]{butkovsky2017invariant} or \cite[Theorem 3.27]{mao2006stochastic}), it follows from the strong uniqueness or pathwise uniqueness that $(X_t,\Lambda(t))$ is Markov. Furthermore, since $(X_t,\Lambda(t))$ has continuous trajectories, the strong Markov property follows from the Feller property (see, e.g., 
		\cite[Theorem 3.3.1]{revuz2013continuous} ). The proof is complete.
	\end{proof}
	\subsection{Exponential Ergodocity}
	In this subsection, we shall give  the exponential ergodicity for the process $(X_t, \Lambda(t))$ under Wasserstein distance.  Before we proceed, we need to construct a coupling process $\left(X(t), \Lambda(t), X^{\prime}(t), \Lambda^{\prime}(t)\right)$ to be used in the argument of the main theorem.  To this aim, let us define the product probability space
	\[
	(\widetilde{\Omega},\widetilde{\mathscr{F}},\widetilde{\mathbb{P}})=(\Omega\times\Omega,\mathscr{F}\otimes\mathscr{F},\mathbb{P}\times \mathbb{P}).
	\]
	We firstly construct the coupling process $\left(\Lambda(t), \Lambda^{\prime}(t)\right)$ of the discrete component.  Note that the  discrete  component  has finite states. Hence,  we define the basic coupling   (see, e.g., \cite[p.11]{chen2004markov})  for $(\Lambda(t), \Lambda^{\prime}(t))$.  For every measurable function $g$ on $\mbb{S} \times \mbb S$, define \begin{equation}\label{eq:switchcoupling}
		\begin{aligned}
			\widetilde{Q} g(k, l)=& \sum_{m\in\mbb S}\left(q_{k m}-q_{l m}\right)^{+}(g(m, l)-g(k, l)) +\sum_{m\in\mbb S}\left(q_{l m}-q_{k m}\right)^{+}(g(k, m)-g(k, l)) \\
			&+\sum_{m\in\mbb S} q_{k m} \wedge q_{l m}(g(m, m)-g(k, l)).
		\end{aligned}
	\end{equation}
	Then, we construct the coupling process   \( (X(t), X^{\prime}(t)) \) of the continuous component.   For $(\varphi, k, \psi, l) \in E\times E$,  set two \( 2 d \times 2 d \) matrices as follows:
	\[
	\sigma(\varphi, \psi, k, l)=\left(\begin{array}{cc}
		\sigma(\varphi, k) & 0 \\
		0 & \sigma(\psi, l)
	\end{array}\right), \quad\sigma(\varphi, \psi, k)=\left(\begin{array}{cc}
		\sigma(\varphi, k) & 0 \\
		\sigma(\psi, k) & 0
	\end{array}\right).
	\]
	Moreover, set
	\[
	\sigma\left(t, \varphi, \psi, \Lambda(t), \Lambda^{\prime}(t)\right)=\mathbf{1}_{[0, \tau)}(t) \sigma\left(\varphi, \psi, \Lambda(t), \Lambda^{\prime}(t)\right)+\mathbf{1}_{[\tau, \infty)}(t) \sigma(\varphi, \psi, \Lambda(t)),
	\]
	where $\tau=\inf \left\{t \geq 0 ; \Lambda(t)=\Lambda^{\prime}(t)\right\}$ be the coupling time of $\left(\Lambda(t), \Lambda^{\prime}(t)\right)$.  Let the coupling process \( (X(t), X^{\prime}(t)) \) satisfy
	\begin{equation}\label{eq:pijingzhanji1}
		\mathrm{d}\left(\begin{array}{c}
			X(t)-G(X_t,\Lambda(t)) \\
			X^{\prime}(t)-G(X_t^{\prime},\Lambda^{\prime}(t))
		\end{array}\right)=\left(\begin{array}{c}
			b\left(X_{t}, \Lambda(t)\right) \\
			b\left(X^{\prime}_{t}, \Lambda^{\prime}(t)\right)
		\end{array}\right) \mathrm{d} t+\sigma\left(t, X_{t}, X^{\prime}_{t}, \Lambda(t), \Lambda^{\prime}(t)\right) \mathrm{d} \widetilde{W}(t),
	\end{equation}
	where \( \widetilde{W}(t) \) is a \( 2 d \)-dimensional Brownian motion independent of \( \left(\Lambda(t), \Lambda^{\prime}(t)\right) \). It is easy to check that $(X(t), X^{\prime}(t))$ determined by  \eqref{eq:pijingzhanji1} is the independent coupling on $[0, \tau)$ and the basic coupling on $[\tau,\infty)$. Namely, before $\Lambda(t)$ and $\Lambda^{\prime}(t)$ are coupled together, $X(t)$ and $ X^{\prime}(t) $ run independently, whereas from $\tau$ onward, $X(t)$ and $ X^{\prime}(t) $ couple each other in the basic coupling manner.

	Finally, we introduce some basic definition and properties on the Wasserstein distance used in this work, and refer the reader to the book \cite{chen2004markov,Villani2008OptimalTO} for more discussions on this well studied topic. Define the distance  $d$  on  $E$ :$$d((\varphi, k),(\psi, l)): = \|\varphi-\psi\|_{r}+\mathbf{1}_{\{k\neq l\}}, \quad\text{for all }(\varphi, k),(\psi, l) \in E.$$ Therefore,  $(E, d)$  is a polish space. Then, given $p\geq 1$, we define the  $L^p$-Wasserstein (or Kantorovich) distance between two probability measures  $\mu, \nu \in   \mathscr{P}(E)$ as follows:  
	$$
	\mathbb{W}_p(\mu, \nu) = \inf _{\pi \in \Pi(\mu, \nu)} \left(\int_{E\times E} [d((\varphi, k),(\psi, l))]^p\mrm d\pi((\varphi,k),(\psi, l))\right)^{\frac{1}{p}},
	$$
	where  $\Pi(\mu, \nu)$  is the collection of probability measures on $E\times E$ having $\mu$ and $\nu$ as marginals. Equipped with $\mathbb{W}_p$, the totality $\msr P_p(E)$ of probability measures having the finite moment of order $p$ becomes a complete metric space. Moreover, recall that the optimal transport cost between the two measures \( \mu, \nu \in  \mathscr{P}(E)\) is defined by:
	\[
	C(\mu, \nu) = \inf_{ \pi \in \Pi(\mu, \nu) }\int c((\varphi, k),(\psi, l)) d\pi((\varphi,k),(\psi, l)),
	\]
	where $c((\varphi, k),(\psi, l))$ is the cost for transporting one unit of mass from $(\varphi, k)$ to $(\psi, l)$.  Hence, $\mathbb{W}_p(\mu,\nu)=(C(\mu,\nu))^{\frac{1}{p}}$ for the cost function $c((\varphi, k),(\psi, l))=d^p((\varphi, k),(\psi, l))$.
	
	For any $\mu\in\mathscr{P}(E)$, set
	\begin{equation}\label{eq:boziteng}
		\mu P_t(\cdot):=\int_{E}\mu(\d\varphi,\d k)\delta_{(\varphi, k)} P_{t}(\cdot).
	\end{equation}
	Moreover, a probability measure $\mu\in\mathscr{P}(E)$ is said to be invariant with respect to $P_t$, if and only if $\mu P_t=\mu$ for any $t\geq 0$. If $\mu$ is an invariant measure of $(X_t,\Lambda(t))$ and the initial data $(X_0,\Lambda(0))$ is distributed as ${\mu}$, then for any ${t>0}$, the distribution of $(X_t,\Lambda(t))$ is ${\mu P_{t}}$, i.e., the law of $(X_t,\Lambda(t))$ is invariant under time translation. 
	
	Now we give the main result in this section.
	\begin{theorem}\label{thm:first}
	For any $p\geq 1$,  suppose that assumptions of Lemma \ref{thm:opao1} are satisfied with $q:=(p+p_0)\vee 2$ for some positive constant $p_0$.
		 Then, there exist two positive constants $C$, $ \varrho$ and a unique invariant probability measure  $\mu\in \msr P_p(E)$ such that $\mu P_t=\mu$ for every $t\geq 0$, and 
		\begin{equation*}
			\mathbb{W}_p\left(\delta_{(\xi, i)} P_{t}, \mu\right) \leq C(1+\|\xi\|_r)\e^{-\varrho t},\quad \text{for any }(\xi, i)\in E.
		\end{equation*}
	\end{theorem}
	\begin{proof}
		The proof is divided into three steps.
		
		\textit{Step 1.} Exponential convergence rate of $C(\delta_{(\xi, i)} P_{t},\delta_{(\eta, j)} P_{t})$.  Let the cost function $c((\varphi, k),(\psi, l))$ be $d^p((\varphi, k),(\psi, l))$. Fix the initial data of the coupling process $(X_t , \Lambda(t) , X^{\prime}_t , \Lambda^{\prime}(t) ) $ to be $(\xi, i,\eta, j)$ with $i\neq j$.  We shall give the estimation for the optimal transport cost $C(\delta_{(\xi, i)} P_{t},\delta_{(\eta, j)} P_{t})$.  Since $\mathbb{S}$ is a finite set, then  $\widetilde{Q}$ defined by \eqref{eq:switchcoupling} is irreducible.  It is well known that there exists a constant $\theta>0$ such that 
		\begin{equation}\label{eq:hsreh}
			\widetilde{\mathbb{P}}(\tau>t) \leq \e^{-\theta t},\qquad t\geq 0.
		\end{equation}
		With the aid of Corollary \ref{cor:tuishitou}, the coupling process $(X_t , \Lambda(t) , X^{\prime}_t , \Lambda^{\prime}(t) ) $ has the strong Markov property.  Using similar arguments to those in Theorem \ref{thm:opaoooo}  and Theorem \ref{thm:opao2}, we can obtain the following results, 
		\begin{equation}
			\widetilde{\mbb{E}}\|X_t\|_r^{q}\leq  {C}(1+\|\xi\|_{r}^{q})\text{ and } \widetilde{\mbb{E}}\|X^{\prime}_t\|_r^{q}\leq  {C}(1+\|\eta\|_{r}^{q}),
		\end{equation}
	and 
	\begin{equation}
		\widetilde{\mbb{E}}\|X_t-X^{\prime}_t\|_r^{p} \leq C \|\xi-\eta\|_{r}^{p}\e^{-\lambda t},
	\end{equation}	
		respectively. Furthermore, by \eqref{eq:hsreh}, we get
		\begin{equation*}
			\begin{aligned}
				\widetilde{\mathbb{E}}&[d((X_{t},\Lambda(t)),(X^{\prime}_{t},\Lambda^{\prime}(t)))]^p \\
				&\leq  C_p\widetilde{\mathbb{E}}\left[\|X_{t}-X^{\prime}_{t}\|_{r}^p\mathbf{1}_{\{\tau\leq pt/q\}}\right]+C_p\widetilde{\mathbb{E}}\left[\left(\left\|X_{t}-X^{\prime}_{t}\right\|_{r}^p+\mathbf{1}_{\{\Lambda(t)\neq\Lambda^{\prime}(t)\}}\right)\mathbf{1}_{\{\tau>pt / q\}}\right]\\
				&\leq C_p\widetilde{\mathbb{E}}\left[\widetilde{\mathbb{E}}\left(\left\|X_{t}-X^{\prime}_{t}\right\|_{r}^p|\widetilde{\msr F}_{\tau}\right)\mathbf{1}_{\{\tau\leq pt/q\}}\right]+C_p\widetilde{\mathbb{P}}(\tau>pt/q)^{1-\frac{p}{q}}\left(\widetilde{\mathbb{E}}\|X_t-X^{\prime}_t\|_r^{q}+1 \right)^{\frac{p}{q}}\\
				&\leq C_p\widetilde{\mathbb{E}}\left[\widetilde{\mathbb{E}}\left(\left\|X_{t-\tau}^{X_{\tau},\Lambda(\tau)}-(X^{\prime})_{t-\tau}^{X^{\prime}_{\tau},\Lambda(\tau)}\right\|_{r}^p\right)\mathbf{1}_{\{\tau\leq pt/q\}}\right]+C_p\widetilde{\mathbb{P}}(\tau>pt/q)^{1-\frac{p}{q}}\left(1+\widetilde{\mathbb{E}}\|X_t\|_r^{q}+\widetilde{\mathbb{E}}\|X^{\prime}_t\|_r^{q} \right)^{\frac{p}{q}}\\
				&\leq  C_p\widetilde{\mathbb{E}}\left[\|X_{\tau}-X^{\prime}_{\tau}\|_r^p \e^{-\frac{p}{q}\lambda(t-\tau)} \mathbf{1}_{\{\tau\leq pt/q\}}\right]+C_p\e^{-\frac{p(q-p)}{q^2}\theta t}\left(1+\|\xi\|_r^{q}+\|\eta\|_r^q\right)^{\frac{p}{q}}\\
				&\leq  C_p \e^{-\frac{p(q-p)}{q^2}\lambda t}\widetilde{\mathbb{E}}\left(\|X_{\tau}\|_r^p+\|X^{\prime}_{\tau}\|_r^p\right)+ C_p\e^{-\frac{p(q-p)}{q^2}\theta t}\left(1+\|\xi\|_r^p+\|\eta\|_r^p\right)\\
				&\leq  C_p(1+\|\xi\|_r^p+\|\eta\|_r^p)\e^{-\vartheta t},
			\end{aligned}
		\end{equation*}
		where $\widetilde{\msr F}_t=\sigma(X_s , \Lambda(s) , X^{\prime}_s , \Lambda^{\prime}(s), s\leq t)$ and $\vartheta=\frac{p(q-p)}{q^2}(\lambda\wedge\theta )$. This implies
		\begin{equation}\label{eq:kuku}
			C(\delta_{(\xi, i)} P_{t},\delta_{(\eta, j)} P_{t})\leq C_p(1+\|\xi\|_r^p+\|\eta\|_r^p)\e^{-\vartheta t}.
		\end{equation}
		
		\textit{Step 2.} Existence and Uniqueness of Invariant Measure.  For fixed $(\eta_0,j_0)$, we have
		\[
		\begin{aligned}
			\int_{E}d^p((\eta,j),(\eta_0,j_0))\delta_{(\xi, i)}P_t (\d \eta,\d j)
			&\leq C_p \mathbb{E}\|X_t^{\xi,i}-\eta_0\|_r^p+C_p\mathbb{P}(\Lambda^i(t)\neq j_0)\\
			&\leq C_p+C_p\left(\sup_{t\geq 0}\mathbb{E}\|X_t^{\xi,i}\|_r^p+\|\eta_0\|_r^p\right).
		\end{aligned}
		\]
		It follows from  Theorem \ref{thm:opaoooo}  that \(\{\delta_{(\xi, i)}P_t \}_{t\geq 0}\subset \mathscr{P}_{p}(E) \). We claim that $\{\delta_{(\xi, i)}P_t \}_{t\geq 0}$ is a Cauchy sequence with respect to the Wasserstein distance \( \mathbb{W}_{p} \). Indeed, due to \eqref{eq:kuku} and the convexity of the  optimal cost (see, e.g., \cite[Theorem 4.8]{Villani2008OptimalTO}), we deduce that for any two probability measures \( \nu_1, \nu_2 \in \) \( \mathscr{P}(E) \),
		\begin{equation}\label{eq:rizi}
			\begin{aligned}
				C\left(\nu_1 P_{t} , \nu_2 P_{t}\right) 
				&= C\left(\int_{E}\nu_1(\d\xi,\d i)\delta_{(\xi, i)} P_{t},\int_{E}\nu_2(\d\eta,\d j)\delta_{(\eta, j)} P_{t}\right)\\
				&=C\left(\int_{E\times E}\delta_{(\xi, i)} P_{t}\mrm d\pi((\xi,i),(\eta, j)),\int_{E\times E}\delta_{(\eta, j)} P_{t}\mrm d\pi((\xi,i),(\eta, j))\right)\\
				&\leq \int_{E\times E} C\left(\delta_{(\xi, i)} P_{t},\delta_{(\eta, j)} P_{t}\right) \mrm d\pi((\xi,i),(\eta, j)) \\
				&\leq  C_p\int_{E\times E}(1+\|\xi\|_r^p+\|\eta\|_r^p)\e^{-\vartheta t}\mrm d\pi((\xi,i),(\eta, j))  \\
				& \leq  C_p\mathrm{e}^{-\vartheta t}\left(1+\int_{E}\|\xi\|_{r}^p \nu_1(\mathrm{d} \xi, \mathrm{d}i)+\int_{E}\|\eta\|_{r}^p \nu_2(\mathrm{d} \eta , \mathrm{d}j)\right),
			\end{aligned}
		\end{equation}
		where $\pi$ is a probability measure on $E\times E$ having $\nu_1$ and $\nu_2$ as marginals. So, by Theorem \ref{thm:opaoooo}, it holds that
		\[
		\begin{aligned}
			\mathbb{W}_{p}\left(\delta_{(\xi, i)}P_{t+s}, \delta_{(\xi, i)}P_t \right) &=\mathbb{W}_p\left(\big(\delta_{(\xi, i)}P_s\big)P_t, \delta_{(\xi, i)}P_t \right) \\
			&=\left(C\left(\big(\delta_{(\xi, i)}P_s\big)P_t, \delta_{(\xi, i)}P_t \right)\right)^{\frac{1}{p}}\\
			&\leq C_p\mathrm{e}^{-\vartheta t/p} \left(1+\int_{E}\|\eta\|_{r}^p\delta_{(\xi, i)} P_{s}(\mathrm{d} \eta, \mrm d j)+\|\xi\|_{r}^p\right)^{\frac{1}{p}} \\
			&\leq  C_p\mathrm{e}^{-\vartheta t/p} \left(1+\|\xi\|_{r}+\sup_{t\geq 0}\left(\E\|X_t^{\xi,i}\|_r^q\right)^{\frac{1}{q}} \right)\\
			&\leq C_p \left(1+\|\xi\|_{r}\right) \mathrm{e}^{-\vartheta t/p},
		\end{aligned}
		\]
		for any \( s, t>0 \). Recall that \( \left(\mathscr{P}_{p}(E), \mathbb{W}_p\right) \) is a complete metric space, then the sequence $\delta_{(\xi, i)}P_{t}$ converges in $\mathbb{W}_p$-metric to some probability measure $\mu\in \mathscr{P}_{p}(E)$ as $t\to\infty$. 
		Moreover, it follows from the triangle inequality for $\mathbb{W}_p$ and \eqref{eq:rizi} that for all $\nu\in \mathscr{P}_p(E)$, $\nu P_{t_k}$ converges in $\mathbb{W}_p$-metric to $\mu.$ Therefore,  letting $\nu=\delta_{(\xi, i)}P_s$ implies that $\delta_{(\xi, i)}P_sP_{t}$ converges in $\mathbb{W}_p$-metric to $\mu$, hence converges weakly to $\mu$.
		
		On the other hand,  by Corollary \ref{cor:tuishitou}, the process $(X_t,\Lambda(t))$ has the Feller property. Hence, for any $s\geq 0$, $\delta_{(\xi, i)}P_{t}P_s$ converges weakly to $\mu P_s$. By the Markov property of $(X_t,\Lambda(t))$, we have $\delta_{(\xi, i)}P_{t}P_s=\delta_{(\xi, i)}P_sP_{t}$. In all,  the uniqueness of the weak convergence implies that for each $s \geq  0$, $\mu P_s = \mu$. Hence, $\mu$ is the unique invariant measure of the process $(X_t, \Lambda(t))$.
		
		\textit{Step 3.}  Exponential Ergodicity.  It follows from \eqref{eq:boziteng} and  \eqref{eq:rizi} that for any $t\geq 0$,
		\begin{equation*}
			{\begin{aligned}
					\mathbb{W}_p\left(\delta_{(\xi, i)} P_{t}, \mu\right) &=\mathbb{W}_p\left(\delta_{(\xi, i)} P_{t}, \mu P_{t}\right) \\
					&=\left(C\left(\delta_{(\xi, i)} P_{t}, \mu  P_{t}\right)\right)^{\frac{1}{p}}\\
					&\leq  C_p\mathrm{e}^{-\vartheta t/p}\left(1+\|\xi\|_{r}^p+\int_{E}\|\eta\|_{r}^p \mu (\mathrm{d} \eta , \mathrm{d}j)\right)^{\frac{1}{p}}\\
					&\leq C_p(1+\|\xi\|_{r})\mathrm{e}^{-\vartheta t/p},
			\end{aligned}}
		\end{equation*}
		where we have used the fact $\mu\in \mathscr{P}_{p}(E)$. Till now, the proof is complete.
	\end{proof}
	\begin{remark}
		Frist, we make a brief remark on the choice of \(q\). When \(1\leq p<2\), it suffices to take \(q=2\); when \(p\geq2\), we only need \(q\) to be slightly larger than \(p\). Moreover, if we set $G\equiv0$, then RNSFDEs simplify to to regime-switching stochastic functional differential equations(RSFDEs) with infinite delay. In addition, we can derive the corresponding exponential ergodicity in Wasserstein distance for RSFDEs with a finite state space. In particular, by choosing $p=1$ and $p_0=1$, there exist two positive constants $C$, $ \varrho$ and a unique invariant probability measure  $\mu\in \msr P_1(E)$ such that $\mu P_t=\mu$ for every $t\geq 0$, and 
		\begin{equation*}
			\mathbb{W}_1\left(\delta_{(\xi, i)} P_{t}, \mu\right) \leq C(1+\|\xi\|_r)\e^{-\varrho t},\quad \text{for any }(\xi, i)\in E.
		\end{equation*}
		This result has been demonstrated in \cite[Theorem 5]{li2021convergence}. Thus, \cite[Theorem 5]{li2021convergence} can be considered a specific instance of Theorem \ref{thm:first}. 
	\end{remark}
	
	\section{Markovian switching in an infinite state space}\label{sec:wuqiong}
	In this section, we consider the RNSFDE $\left(X_t, \Lambda(t)\right)$ with $N=\infty$.  Owing to  the infinite state space, the transition rate matrix $Q$ of the Markov chain $(\Lambda(t))$ is an infinite dimensional matrix. The method used in Section \ref{sec:4} does not work any more. In order to overcome this difficulty, we adopt the approach proposed in \cite{SHAO2015SPA} and construct a finite dimensional transition rate matrix $Q^F$ so that $Q$ can be controlled by $Q^F$.  
	Let us provide a brief description of the approach as follows. 
	
	First, we divide $\mathbb{S}$ into finite subsets according to ${\alpha(\cdot)}$. Precisely, choose a finite partition $\mathcal{M}$ of ${(-\infty, \bar{\alpha}]}$ of size ${m}$ with ${m \geq 1}$, that is,
	\[
	\mathcal{M}:=\left\{-\infty=: i_{0}<i_{1}<\cdots<i_{m}:=\bar{\alpha}\right\} .
	\]
	Corresponding to $\mathcal{M}$, there is a finite partition of ${\mathbb{S}}$, denoted by ${F:=\left\{F_{1}, \ldots, F_{m}\right\}}$, where
	\[
	F_{n}=\left\{l \in \mathbb{S} ; \alpha(l) \in\left(i_{n-1}, i_{n}\right]\right\},  \;n=1, \ldots, m .
	\]
	We assume that each ${F_{n}}$ is not empty, otherwise, we can delete some points in the partition $\mathcal{M}$ to ensure it.  Let
	\begin{equation}\label{eq:xinffu}
		\alpha^{F}(n)=\sup _{k \in F_{n}} \alpha(k),  \;\beta^{F}(n)=\sup _{k \in F_{n}} \beta(k),~ n=1, \ldots, m,
	\end{equation}
	so ${\alpha^{F}(n)<\alpha^{F}(n+1)}$ and ${\beta^{F}(n)<\beta^{F}(n+1)}$. Set ${Q^{F}=\left(q_{k l}^{F}\right)}_{m\times m}$ be a new ${Q}$-matrix on the state space ${\{1,2, \ldots, m\}}$ corresponding to ${F}$ defined by
	\begin{equation}\label{eq:xinleng}
		q_{k l}^{F}=\inf _{j_1 \in F_{k}} \sum_{j_2 \in F_{l}} q_{j_1 j_2},  l>k ; \quad q_{k l}^{F}=\sup _{j_1 \in F_{k}} \sum_{j_2 \in F_l} q_{j_1 j_2}, l<k ;\quad  q_{k k}^{F}=-\sum_{l \neq k} q_{k l}^{F}.
	\end{equation}
	Moreover, since each ${F_{n}}$ is nonempty, we obtain ${0 \leq q_{k l}^{F} \leq \sup _{k \in \mathbb{S}} q_{k}<\infty, l \neq k}$, which ensures that $Q^F$ is well-defined.

	\subsection{Boundedness of $X_t$}
	First, we establish the uniform boundedness for the process $Y^{\xi, i}(t)$ with the initial data $(\xi,i)\in  E$. Here, we introduce an alternative method which involves the use of Lyapunov functions and $M$-matrices as the primary techniques. The theory of \(M\)-matrices can be found in  \cite{berman1994nonnegative} or \cite[Chapter 2.6]{mao2006stochastic}.
	
	Let us introduce the  operator associated with the RNSFDE. Let ${C^{2}\left( \mathbb{R}^{n}\times \mbb S ; \mathbb{R}^{+}\right)}$ denote the family of all non-negative functions ${V(x,k)}$ on $\mathbb{R}^{n}\times \mbb S$ which are continuously twice differentiable in ${x}$; moreover, define
	\[
	{V_{x}(x,k)=\left(\frac{\partial V(x,k)}{\partial x_{1}}, \ldots, \frac{\partial V(x,k)}{\partial x_{n}}\right)},\quad
	{V_{x x}(x,k)=\left(\frac{\partial^{2} V(x,k)}{\partial x_{l_1} \partial x_{l_2}}\right)_{n \times n}}.
	\]
	Then, we can define an operator \( \mathcal{L} V \) from $ E $ to \( \mathbb{R} \) by
	\[
	\begin{aligned}
		(\mathcal{L} V)(\varphi, k) &=\langle V_{x}(\varphi(0)-G(\varphi,k), k), b(\varphi,k)\rangle +\frac{1}{2} \operatorname{trace}\left[\sigma^{\top}(\varphi, k) V_{x x}(\varphi(0)-G(\varphi,k), k) \sigma(\varphi, k)\right]\\
		&\quad+\sum_{l\in\mbb S} q_{k l} V(\varphi(0)-G(\varphi,k), l) .
	\end{aligned}
	\]
	It should be emphasized that the operator \( \mathcal{L} V \) (thought as a single notation rather than \( \mathcal{L} \) acting on \( V) \) is defined on $ E$, although \( V \) is defined on \( \mathbb{R}^{n} \times \mathbb{S} \). 
	
	\begin{lemma}\label{thm:shenchen}
		Let assumption (\hyperlink{A0}{A0}), (\hyperlink{H1}{H1}), and (\hyperlink{H2}{H2}) hold. Suppose that there exist a function \( V \in C^{2}\left(\mathbb{R}^{n} \times \mathbb{S} ; \mathbb{R}_{+}\right) \) and positive constants $ p$, $ c_{1} $, $ c_{2} $, $\lambda_0$, $\lambda_{1} $, $ \lambda_{2} $  such that \( \lambda_{1}>\frac{\lambda_{2}}{(1-\kappa_1)^{p}}\),
		\begin{equation}\label{eq:sinian}
			c_{1}|x|^{p} \leq V(x, k) \leq c_{2}|x|^{p}
		\end{equation}
		for \( (x, k) \in \mathbb{R}^{n} \times \mathbb{S} \) and
		\begin{equation}\label{eq:huiyi}
			(\mathcal{L} V)(\varphi, k) \leq \lambda_0-\lambda_{1}|\varphi(0)-G(\varphi,k)|^{p}+\lambda_{2} \int_{-\infty}^{0} |\varphi(\theta)|^{p} \rho(\d \theta)
		\end{equation}
		for \( (\varphi,  k) \in C_r \times \mathbb{S} \). Then there exists a constant $C>0$  such that for any   $(\xi,i)\in C_r\times\S$ and  $t\geq 0 $,
		\begin{equation}
			\mbb{E}\left|Y^{\xi, i}(t)\right|^{p}\leq  C(1+\|\xi\|_{r}^{p}).
		\end{equation}	
	\end{lemma}
	\begin{proof}
		Write $Y^{\xi, i}(t)=Y(t)$ simply.  Let $\lambda\in (0,pr)$ be a positive constant determined as follows.  Applying the generalized Itô formula (see e.g., \cite[Theorem 1.45]{mao2006stochastic}) to $\e^{\lambda t}V(Y(t),\Lambda(t))$ and then using \eqref{eq:huiyi}, we can show that,
		\begin{equation*}
			\begin{aligned}
				\e^{\lambda t} \E V(Y(t), \Lambda(t))&=\E V(Y(0), \Lambda(0))+\E \int_{0}^{t} \e^{\lambda u}\left(\lambda V(Y(u), \Lambda(u))+(\mathcal{L}V)(X_{u}, \Lambda{(u)})\right) \d u \\
				&\leq  c_2 2^p\|\xi\|_r^p+\lambda c_{2}\E \int_{0}^{t} \e^{\lambda u} |Y(u)|^{p} \d u+\frac{\lambda_0}{\lambda} \e^{\lambda t}-\lambda_{1}\E \int_{0}^{t} \e^{\lambda u} |Y(u)|^{p} \d u\\
				&\quad+\lambda_{2} \E \int_{0}^{t} \int_{-\infty}^{0} \e^{\lambda u}|X(u+\theta)|^{p} \rho(\d \theta)\d u, \\
			\end{aligned}
		\end{equation*}
		for any  $t\geq 0$.  Similar to \eqref{eq:xixifu}, we have 
		\begin{equation*}
			\int_{0}^{t} \int_{-\infty}^{0} \e^{\lambda u}|X(u+\theta)|^{p} \rho(\mathrm{d} \theta) \mathrm{d} u 
			\leq \frac{\delta_{pr}\left(\rho\right)}{1-\kappa_1}\|\xi\|_{r}^{p}t+\frac{1}{(1-\kappa_1)^{p}} \int_{-\infty}^{0} \e^{-\lambda \theta} \rho(\d \theta) \int_{0}^{t}\e^{\lambda u}|Y(u)|^p\mathrm{d}u.
		\end{equation*}
		Combining the above computation implies
		\begin{equation}\label{eq:chengh}
			\begin{aligned}
				\e^{\lambda t} &\E V(Y(t), \Lambda(t))\\
				&\leq C_p(1+t)\|\xi\|_{r}^{p}+\frac{\lambda_0}{\lambda} \e^{\lambda t}+\left(\lambda c_2-\lambda_{1}+\frac{\lambda_{2}}{(1-\kappa_1)^{p}}\int_{-\infty}^{0} \e^{-\lambda \theta} \rho(\d \theta)  \right)\E \int_{0}^{t} \e^{\lambda u} |Y(u)|^{p} \d u.
			\end{aligned}
		\end{equation}
		Since \( \lambda_{1}>\lambda_{2}/(1-\kappa_1)^{p}\), there exists a positive constant $\varsigma\in(0,pr)$  satisfying
		\begin{equation}\label{Eq:1014:2}
			\varsigma c_2-\lambda_{1}+\frac{\lambda_{2}}{(1-\kappa_1)^{p}}\int_{-\infty}^{0} \e^{-\varsigma\theta} \rho(\d \theta) \leq 0.
		\end{equation}
		Choosing $\lambda= \varsigma$, the aforementioned computations remain valid. Moreover, we have
		\[
		\E |Y(t)|^p\leq C_p (1+t\e^{-\lambda t})\|\xi\|_{r}^{p} +\frac{\lambda_0}{\lambda c_1}\leq C_p(1+\|\xi\|_{r}^{p} ),
		\]
		where we have used  \eqref{eq:sinian}. The proof is now complete.
	\end{proof}
	\begin{proposition}\label{Pro:0929}
		Let $p\geq 2$ and assumptions (\hyperlink{A0}{A0})-(\hyperlink{A3}{A3}),  (\hyperlink{H1}{H1}) and (\hyperlink{H2}{H2})  hold. Suppose further that $$\mathcal{A}:=-\left(p\operatorname{diag}\left(\alpha^{F}(1), \ldots, \alpha^{F}(m)\right)+Q^{F}\right) H_{m} $$ is a nonsingular $M$-matrix, where  $\alpha^F$ and $Q^F$ are defined by \eqref{eq:xinffu} and \eqref{eq:xinleng} respectively, and
		\begin{equation}\label{eq:mingji}
			H_{m}=\left(\begin{array}{ccccc}
				1 & 1 & 1 & \cdots & 1 \\
				0 & 1 & 1 & \cdots & 1 \\
				\vdots & \vdots & \vdots & \ddots & \vdots \\
				0 & 0 & 0 & \cdots & 1
			\end{array}\right)_{m \times m}.
		\end{equation}
		Then, we have
		$$u^{F}=\left(u^{F}(1), \ldots,u^{F}(m)\right)^{\top} :=\mathcal{A}^{-1}\vec{1}\gg 0. $$
		Let $ v^{F}=H_{m} u^{F}$. If 
		\begin{equation}\label{Eq:1014:1}
			( \beta^F(n)+\gamma_p) v^F(n)<\frac{(1-\kappa_1)^p}{2+(p-2)(1-\kappa_1)^p}
		\end{equation}
		for all $1\leq n\leq m$, then Lemma \ref{thm:shenchen} is applicable.
	\end{proposition}
	\begin{proof}
		Note that $v^F=H_mu^F$, then 
		\begin{equation}\label{Eq:0929}
			\left(p\operatorname{diag}\left(\alpha^{F}(1), \ldots, \alpha^{F}(m)\right)+Q^{F}\right) v^{F} = -\vec{1}.
		\end{equation}
		Moreover,
		\[
		v^{F}({n})=u^{F}(n)+\cdots+u^{F}(m), \quad n=1, \ldots, m .
		\]
		Hence, $v^{F}({n+1})<v^{F}({n})$, $n=1, \ldots, m-1 $, and ${v^{F} \gg 0}$. Let us extend the vector ${v^{F}}$ to a vector on ${\mathbb{S}}$ by setting ${v({k})=v^{F}({n})}$, if ${k \in F_{n}}$. Thus, as in  \cite[Theorem 4.1]{SHAO2015SPA}, we have 
		\begin{equation}\label{Eq:1}
			(Q v)(k)\leq \left(Q^{F} v^{F}\right)(n) \text{ for }  k \in F_{n} .
		\end{equation}
		
		Now, we shall show that \eqref{eq:sinian} and  \eqref{eq:huiyi} hold. Define a \( C^{2} \)-function \( V: \mathbb{R}^{n} \times \mathbb{S} \rightarrow \mathbb{R}_{+} \) by
		\[
		V(x,k)=v(k)|x|^{p}
		\]
		Clearly, \( V (x,k)\) obeys \eqref{eq:sinian} with \( c_{1}=\inf_{l \in \mathbb{S}} v(l)=\min_{1\leq j\leq m}v^F(j) >0\) and \( c_{2}=\sup_{l \in \mathbb{S}} v(l)=\max_{1\leq j\leq m}v^F(j)>0\). It follows from (\hyperlink{A3}{A3}) that \eqref{eq:qianya} still holds with 
		\[
		C(\varepsilon)=\frac{1}{4\varepsilon} \sup _{l \in \mathbb{S}}|b(\textbf{0}, l)|^{2}+\frac{p-1}{2}\left(1+\frac{\gamma_p}{\varepsilon}\right)\sup _{l \in \mathbb{S}}\|\sigma(\textbf{0}, l)\|_{\rm H S}^{2}.
		\]
		To verify \eqref{eq:huiyi}, we compute directly the operator \( \mathcal{L} V \) from \(  C_r \times \mathbb{S} \) to \( \mathbb{R} \) by \eqref{eq:qianya} as follows. For any $k\in \mbb S$, there exists a positive integer $n$, $1\leq n\leq m$, such that $k\in F_n$. Then, by \eqref{eq:xinffu},  \eqref{Eq:0929} and \eqref{Eq:1}, we get
		\begin{equation*}
			\begin{aligned}
				(\mathcal{L}& V)(\varphi,k)\\
				&\leq pv(k)|\varphi(0)-G(\varphi,k)|^{p-2}\left(  \langle\varphi(0)-G(\varphi,k), b(\varphi,k)\rangle+\frac{p-1}{2} \|\sigma(\varphi,k)\|_{\text{HS}}^{2} \right)\\
				&\quad+\left(\sum_{l\in\S} q_{k l} v(l)\right)|\varphi(0)-G(\varphi,k)|^{p}\\
				&\leq  pC(\varepsilon)v(k)|\varphi(0)-G(\varphi,k)|^{p-2}+\left(pv(k)(\alpha(k)+\varepsilon)+\sum_{l\in\S} q_{kl}v(l)\right)|\varphi(0)-G(\varphi,k)|^{p}\\
				&\quad+pv(k)(\beta(k)+\gamma_p+\varepsilon)|\varphi(0)-G(\varphi,k)|^{p-2}\int_{-\infty}^{0} |\varphi(\theta)|^{2} \rho(\d \theta) \\
				&\leq  pC(\varepsilon)v^F(n)|\varphi(0)-G(\varphi,k)|^{p-2}+\left(pv^F(n)(\alpha^F(n)+\varepsilon)+\sum_{1\leq j\leq m} q^F_{nj}v^F(j)\right)|\varphi(0)-G(\varphi,k)|^{p}\\
				&\quad+p(\beta^F(n)+\gamma_p+\varepsilon)v^F(n)|\varphi(0)-G(\varphi,k)|^{p-2}\int_{-\infty}^{0} |\varphi(\theta)|^{2} \rho(\d \theta) \\
				&\leq  C_1(\varepsilon)|\varphi(0)-G(\varphi,k)|^{p-2}+\left(pc_2\varepsilon-1\right)|\varphi(0)-G(\varphi,k)|^{p}\\
				&\quad+p\left(K+c_2\varepsilon\right)|\varphi(0)-G(\varphi,k)|^{p-2}\int_{-\infty}^{0} |\varphi(\theta)|^{2} \rho(\d \theta),
			\end{aligned}
		\end{equation*}
		where \( C_1(\varepsilon)=pc_2C(\varepsilon) \) and $K=\max _{1\leq j\leq m} (\beta^F(j)+\gamma_p) v^F(j)$. Note that
		\[
		\begin{aligned}
			|\varphi(0)-G(\varphi,k)|^{p-2}|\varphi(\theta)|^{2} \leq \frac{p-2}{p}|\varphi(0)-G(\varphi,k)|^{p}+\frac{2}{p}|\varphi(\theta)|^{p},
		\end{aligned}
		\]
		while, for any \( \varepsilon>0 \),
		\[
		C_1(\varepsilon)|\varphi(0)-G(\varphi,k)|^{p-2} \leq\left[\varepsilon^{\frac{2-p}{2}} (C_1(\varepsilon))^{\frac{p}{2}}\right]^{\frac{2}{p}}\left(\varepsilon|\varphi(0)-G(\varphi,k)|^{p}\right)^{\frac{p-2}{p}} \leq C_2(\varepsilon)+\frac{p-2}{p}\varepsilon|\varphi(0)-G(\varphi,k)|^{p},
		\]
		where $C_2(\varepsilon)=\frac{2}{p}\varepsilon^{\frac{2-p}{2}} (C_1(\varepsilon))^{\frac{p}{2}}$. Combining the above  computations yields
		\begin{equation*}
				(\mathcal{L} V)(\varphi,k)
				\leq C_2(\varepsilon)-\lambda_{1}|\varphi(0)-G(\varphi,k)|^p+
				\lambda_{2}\int_{-\infty}^{0} |\varphi(\theta)|^{p} \rho(\d \theta),
		\end{equation*}
		where 
		$$
		\begin{aligned}
			\lambda_{1}&=1-\frac{p-2}{p}\varepsilon-2(p-1)c_2\varepsilon-(p-2)K,\\
			\lambda_{2}&=2\left(K+c_2\varepsilon\right).
		\end{aligned}
		$$
		By \eqref{Eq:1014:1}, we can choose $\varepsilon$ sufficiently small such that $\lambda_{1}>\lambda_{2}/(1-\kappa_1)^{p}$. Hence, all conditions of Lemma \ref{thm:shenchen} have been verified and the conclusion of this proposition follows.
	\end{proof}
	Using a similar argument of Theorem \ref{thm:opaoooo}, it follows that
	\begin{theorem}\label{Eq:1105:4}
		Suppose that assumptions of Proposition \ref{Pro:0929} hold.  Then there exists a constant $C>0$  such that for any   $(\xi,i)\in C_r\times\S$ and  $t\geq 0 $,
		\begin{equation}
			\mbb{E}\|X^{\xi, i}_t\|_r^{p} \leq  {C}(1+\|\xi\|_{r}^{p}).
		\end{equation}
	\end{theorem}
	\subsection{Convergence  of  $X_t$}
	
	In order to establish the convergence of $X_t$, we first give the convergence for the process $Y^{\xi, i}(t)$ with the initial data $(\xi,i)\in E$.   Consider the difference between two processes starting from different initial data, namely
	\begin{equation}\label{Eq:1011}
		\begin{aligned}
			Y^{\xi, i}(t)-Y^{\eta, i}(t) &=  \xi(0)-\eta(0)+\int_{0}^{t}\left[b\left(X^{\xi, i}(s), \Lambda^{i}(s)\right)-b\left(X^{\eta, i}, \Lambda^{i}(s)\right)\right] \d s \\
			&\quad+  \int_{0}^{t}\left[\sigma\left(X^{\xi, i}(s), \Lambda^{i}(s)\right)-\sigma\left(X^{\eta, i}, \Lambda^{i}(s)\right)\right] \d W(s) .
		\end{aligned}
	\end{equation}
	Moreover, we need to introduce a new operator. For a given function \( U \in C^{2}\left(\mathbb{R}^{n} \times \mathbb{S} ; \mathbb{R}_{+}\right) \), we define an operator \( \mathbb{L} U:C_r \times C_r\times \mathbb{S} \rightarrow \mathbb{R} \) associated with equation \eqref{Eq:1011} by
	\[
	\begin{aligned}
		(\mathbb{L} U)&(\varphi, \psi, k) \\
		&=  \langle U_{x}(\varphi(0)-\psi(0)+G(\psi,k)-G(\varphi,k), k),(b(\varphi, k)-b(\psi, k))\rangle\\
		&\quad+\frac{1}{2} \operatorname{trace}\left[(\sigma(\varphi, k)-\sigma(\psi, k))^{\top} U_{x x}(\varphi(0)-\psi(0)+G(\psi,k)-G(\varphi,k), k)(\sigma(\varphi, k)-\sigma(\psi, k))\right] \\
		& \quad+\sum_{l\in\S} q_{kl} U(\varphi(0)-\psi(0)+G(\psi,k)-G(\varphi,k), l)
	\end{aligned}
	\]
	
	\begin{lemma}\label{Eq:1014:3}
		Let assumption (\hyperlink{A0}{A0}) hold. Suppose that there exist a function \( U \in C^{2}\left(\mathbb{R}^{n} \times \mathbb{S} ; \mathbb{R}_{+}\right) \) and positive constants $ p$, $ c_{1} $, $ c_{2} $, $\lambda_{1} $, $ \lambda_{2} $  such that \( \lambda_{1}>\lambda_{2}(1-\kappa_1)^{p}\),
		\begin{equation}\label{Eq:1014:4}
			c_{1}|x|^{p} \leq U(x, k) \leq c_{2}|x|^{p}
		\end{equation}
		for \( (x, k) \in \mathbb{R}^{n} \times \mathbb{S} \) and
		\begin{equation}\label{Eq:1011.2}
			\begin{aligned}
				(\mathbb{L} U)&(\varphi,\psi, k)
				\leq -\lambda_{1}|\varphi(0)-\psi(0)+G(\psi,k)-G(\varphi,k)|^{p}+\lambda_{2} \int_{-\infty}^{0} |\varphi(\theta)-\psi(\theta)|^{p} \rho(\d \theta)
			\end{aligned}
		\end{equation}
		for any \( (\varphi, \psi, k) \in C_r \times C_r\times \mathbb{S} \). Then there exists a constant $C>0$  such that for any   $(\xi,\eta,i)\in C_r\times C_r\times\S$ and  $t\geq 0 $,
		\begin{equation}
			\mbb{E}\left|Y^{\xi, i}(t)-Y^{\eta, i}(t)\right|^{p}\leq  C\|\xi-\eta\|_{r}^{p}\e^{-\lambda t}.
		\end{equation}	
		
	\end{lemma}
	\begin{proof}
		The proof is almost similar to Lemma \ref{thm:shenchen}. Whereas we herein provide an outline of the argument just to make the content self-contained and emphasize some corresponding differences. Let $\lambda\in (0,pr)$ be a positive constant determined as follows.  Applying the generalised Itô formula to $\e^{\lambda t}U(Y^{\xi, i}(t)-Y^{\eta, i}(t),\Lambda^i(t))$ and then using \eqref{Eq:1011.2}, we can show that,
		\begin{equation*}\label{eq:xiangni}
			\begin{aligned}
				\e^{\lambda t} &\E U(Y^{\xi, i}(t)-Y^{\eta, i}(t),\Lambda^i(t))\\
				&=U(Y^{\xi, i}(0)-Y^{\eta, i}(0),\Lambda^i(0))+\E \int_{0}^{t} \e^{\lambda u}\left[\lambda U(Y^{\xi, i}(u)-Y^{\eta, i}(u),\Lambda^i(u))\right.\\
				&\quad\left.+(\mathbb{L}U)(X^{\xi, i}(u),X^{\eta, i}(u),\Lambda^i(u))\right] \d u \\
				&\leq  c_22^p\|\xi-\eta\|_r^p+(\lambda c_{2}-\lambda_1)\E \int_{0}^{t} \e^{\lambda u} |Y^{\xi, i}(u)-Y^{\eta, i}(u)|^{p} \d u\\
				&\quad+\lambda_{2} \E \int_{0}^{t} \int_{-\infty}^{0} \e^{\lambda u}|X^{\xi, i}(u+\theta)-X^{\eta, i}(u+\theta)|^{p} \rho(\d \theta)\d u, \\
			\end{aligned}
		\end{equation*}
		for any  $t\geq 0$. Moreover, we have 
		\begin{equation*}
			\begin{aligned}
				\int_{0}^{t} &\int_{-\infty}^{0} \e^{\lambda u}|X^{\xi, i}(u+\theta)-X^{\eta, i}(u+\theta)|^{p} \rho(\mathrm{d} \theta) \mathrm{d} u \\
				&\leq \frac{\delta_{pr}\left(\rho\right)}{1-\kappa_1}\|\xi-\eta\|_{r}^{p}t+\frac{1}{(1-\kappa_1)^{p}} \int_{-\infty}^{0} \e^{-\lambda \theta} \rho(\d \theta) \int_{0}^{t}\e^{\lambda u}|Y^{\xi, i}(u)-Y^{\eta, i}(u)|^p\mathrm{d}u.
			\end{aligned} 
		\end{equation*}
		Combining the above computation implies
		\begin{equation}
			\begin{aligned}
				\e^{\lambda t} &\E U(Y^{\xi, i}(t)-Y^{\eta, i}(t),\Lambda^i(t))\\
				&\leq C_p\|\xi-\eta\|_{r}^{p}++\left(\lambda c_2-\lambda_{1}+\frac{\lambda_{2}}{(1-\kappa_1)^{p}}\int_{-\infty}^{0} \e^{-\lambda \theta} \rho(\d \theta)  \right)\E \int_{0}^{t} \e^{\lambda u} |Y^{\xi, i}(u)-Y^{\eta, i}(u)|^{p} \d u.
			\end{aligned}
		\end{equation}
		Similarly, choosing $\lambda= \varsigma$ gives
		\[
		\E |Y^{\xi, i}(u)-Y^{\eta, i}(u)|^p\leq C_p \|\xi-\eta\|_{r}^{p}\e^{\lambda t}.
		\]
		The proof is now complete.
	\end{proof}
	\begin{proposition}
		Suppose that assumptions of Proposition \ref{Pro:0929} hold. Then Lemma \ref{Eq:1014:3} is applicable.
	\end{proposition}
	\begin{proof}
		In the sequel, we use the same notations as in Proposition \ref{Pro:0929}. We shall show that \eqref{Eq:1014:4} and  \eqref{Eq:1011.2} are fulfilled. Define a \( C^{2} \)-function \( U: \mathbb{R}^{n} \times \mathbb{S} \rightarrow \mathbb{R}_{+} \) by
		\[
		U(x,k)=v(k)|x|^{p}
		\]
		Hence, \( U (x,k)\) obeys \eqref{Eq:1014:4} with \( c_{1}=\min_{1\leq j\leq m}v^F(j) >0\) and \( c_{2}=\max_{1\leq j\leq m}v^F(j)>0\). For any $k\in \mbb S$, there exists a positive integer $n$, $1\leq n\leq m$, such that $k\in F_n$. Then, by (\hyperlink{A1}{A1}),(\hyperlink{A2}{A2}) and Young's inequality, we get
		\begin{equation*}\label{eq:beijiaer}
			\begin{aligned}
				(\mathbb{L}&U)(\varphi,\psi,k)\\
				&\leq pv(k)|\varphi(0)-\psi(0)+G(\psi,k)-G(\varphi,k)|^{p-2}\\
				&\qquad \times \left(  \langle\varphi(0)-\psi(0)+G(\psi,k)-G(\varphi,k), b(\varphi,k)-b(\psi,k)\rangle+\frac{p-1}{2} \|\sigma(\varphi,k)-\sigma(\psi, k)\|_{\text{HS}}^{2} \right)\\
				&\quad+\left(\sum_{l\in\S} q_{kl} v(l)\right)|\varphi(0)-\psi(0)+G(\psi,k)-G(\varphi,k)|^p\\
				&\leq  \left(pv(k)\alpha(k)+\sum_{l\in\S} q_{kl}v(l)\right)|\varphi(0)-\psi(0)+G(\psi,k)-G(\varphi,k)|^p\\
				&\quad+pv(k)(\beta(k)+\gamma_p)|\varphi(0)-\psi(0)+G(\psi,k)-G(\varphi,k)|^{p-2}\int_{-\infty}^{0} |\varphi(\theta)-\psi(\theta)|^{2} \rho(\d \theta) \\
				&\leq \left(p\alpha^F(n)v^F(n)+\sum_{1\leq j\leq m} q^F_{nj}v^F(j)\right)|\varphi(0)-\psi(0)+G(\psi,k)-G(\varphi,k)|^p\\
				&\quad+p(\beta^F(n)+\gamma_p)v^F(n)|\varphi(0)-\psi(0)+G(\psi,k)-G(\varphi,k)|^{p-2}\int_{-\infty}^{0} |\varphi(\theta)-\psi(\theta)|^{2} \rho(\d \theta) \\
				&\leq  -\left(1-(p-2)K\right)|\varphi(0)-G(\varphi,k)|^{p}+2K\int_{-\infty}^{0} |\varphi(\theta)-\psi(\theta)|^{p} \rho(\d \theta).\\
			\end{aligned}
		\end{equation*}
		Hence, \eqref{Eq:1014:1}  derive  the required assertion. The proof is complete.
	\end{proof}
	Similarly, we have the following result. 
	\begin{theorem}\label{Eq:1105:5}
		Suppose that assumptions of Proposition  \ref{Pro:0929}  hold. Then there exist constants $C,\lambda>0$  such that for any   $\xi$, $\eta \in  C_r$,  $i \in \mbb S $, and  $t\geq 0 $,
		\begin{equation}
			\mbb{E}\|X^{\xi, i}_t-X^{\eta, i}_t\|_r^{p} \leq C \|\xi-\eta\|_{r}^{p}\e^{-\lambda t}.
		\end{equation}	
	\end{theorem}
	\subsection{Exponential Ergodocity}
	In order to establish the exponential ergodicity,	 we also need to construct the coupling process  $\left(X(t), \Lambda(t), X^{\prime}(t), \Lambda^{\prime}(t)\right)$.  Construct the coupling process \( (X(t), X^{\prime}(t)) \) as in \eqref{eq:pijingzhanji1}.  Since \( \S \) is a countable set, we need more assumption on \( \left(\Lambda(t)\right) \) so that the coupling method could be applied.
	
	\noindent(\hypertarget{H3}{H3}) There is a coupling process \( \left(\Lambda(t), \Lambda^{\prime}(t)\right) \) with operator \( \widetilde{Q} \) on \( \S \times \S \). Suppose there is a bounded function \( g \geq 0 \) in the domain of \( \widetilde{Q} \) such that \( g(i, i)=0 \) and
	\[
	\widetilde{Q} g(i, j) \leq-1, \quad i \neq j .
	\]
	
	\begin{remark}
		\eqref{eq:hsreh} is what we need to estimate the Wasserstein distance directly. (\hyperlink{H3}{H3}) provides a sufficient condition to guarantee \eqref{eq:hsreh} hold with $N=\infty$. Indeed, according to \cite[Theorem 5.18]{chen2006eigenvalues},  (\hyperlink{H3}{H3})  implies that for \( 0<\theta<1/\|g\|_{\infty} \),
		\begin{equation*}\label{eq:ising}
			\E[\e^{\theta \tau}]\leq \frac{1}{1-\theta\|g\|_{\infty}},
		\end{equation*}
		where \( \tau \) is the coupling time and $\|g\|_{\infty}:=\sup_{k,l\in\S}|g(k,l)|$. Choosing and fixing a $\theta$ such that $0<\theta<1/1/\|g\|_{\infty}$,  by the Markov inequality, we have
		\begin{equation*}\label{Eq:1019:1}
			\mathbb{P}(\tau>t)=\mathbb{P}(\e^{\theta\tau}>\e^{\theta t})\leq  \frac{\E{\e^{\theta\tau}}}{\e^{\theta t}}\leq   \frac{1}{1-\theta\|g\|_{\infty}}\e^{-\theta t},\; \text{for any }t\geq 0.
		\end{equation*}
		
	\end{remark}

	With aid of  (\hyperlink{H3}{H3}), Theorems  \ref{Eq:1105:4} and  \ref{Eq:1105:5}, we get the following theorem. The proof is almost similar to Theorem \ref{thm:first}, so we omit it here.
	\begin{theorem}\label{thm:secondly}
		For any $p\geq 1$,  suppose that assumptions of Proposition   \ref{Pro:0929}  are satisfied with $q:=(p+p_0)\vee 2$ for some positive constant $p_0$. Then, there exist two positive constant $C$, $ \varrho$ and a unique invariant probability measure  $\mu\in \msr P_p(E)$ such that $\mu P_t=\mu$ for every $t>0$, and 
		\begin{equation*}
			\mathbb{W}_p\left(\delta_{(\xi, i)} P_{t}, \mu\right) \leq C(1+\|\xi\|_r)\e^{-\varrho t},\quad \text{for any }(\xi, i)\in E.
		\end{equation*}
	\end{theorem}
	\begin{remark}
		If we set $G\equiv0$,  we can derive the corresponding exponential ergodicity in Wasserstein distance for RSFDEs with a infinite state space. In particular, by choosing $p_0=p$, there exist two positive constants $C$, $ \varrho$ and a unique invariant probability measure  $\mu\in \msr P_p(E)$ such that $\mu P_t=\mu$ for every $t\geq 0$, and 
		\begin{equation*}
			\mathbb{W}_p\left(\delta_{(\xi, i)} P_{t}, \mu\right) \leq C(1+\|\xi\|_r)\e^{-\varrho t},\quad \text{for any }(\xi, i)\in E.
		\end{equation*}
		This result is established in \cite[Theorem 4.2]{zhang2025exponential}, which thus serves as a specific instance of Theorem \ref{thm:secondly}.

	\end{remark}
	\section{Conclusion}
	
	In summary, this paper demonstrates the exponential ergodicity in Wasserstein distance for  RNSFDEs with infinite delay, in which the Markovian switching may have finite or countably infinite many states. The existence of the neutral term $G$ has caused significant challenges in our research. When $G\equiv0$, the model simplifies to RSFDEs with infinite delay. In this case, (\hyperlink{A0}{A0}) is automatically satisfied. Moreover, we can obtain the exponential ergodicity in Wasserstein distance for RSFDEs. 
	
	Considering a broader framework, the switching component may depend on the historical trajectory of the continuous process. Specifically, the component $(\Lambda(t))$ is a right continuous stochastic process taking values in the state space $\mbb S$ with generator $Q=(q_{kl}(\varphi))_{N\times N}$ given by
	\begin{equation*}
		\mbb {P}\{\Lambda(t+\Delta)=l \mid X_t=\varphi, \Lambda(t)=k \}= \begin{cases}q_{kl}(\varphi) \Delta+o(\Delta), & l \neq k, \\ 1+q_{kk}(\varphi) \Delta+o(\Delta), & l=k,\end{cases}
	\end{equation*}
	provided $\Delta \downarrow 0$, where $q_{kl}(\varphi) \geq 0$ is the transition rate from $k$ to $l$, if $l \neq k$; while $q_{kk}(\varphi)=-\sum_{l \neq k} q_{kl}(\varphi)$. This switching is known as stat-dependent switching or  past-dependent switching. Correspondingly, we can investigate the exponential erogodicity for such processes. However, since the coexistence and interaction of continuous dynamics and switching processes poses significant challenges in studying their properties, there are limited conclusions in this area. This interesting problem is reserved for future research.

	\bibliographystyle{plain}
	\bibliography{chap}

\end{document}